\documentclass[preprint]{elsarticle}
\usepackage[hyperindex,breaklinks,colorlinks=true,linkcolor=black,anchorcolor=black,citecolor=black,filecolor=black,menucolor=black,runcolor=black,urlcolor=blue,pagebackref]{hyperref}
\usepackage[utf8]{inputenc}
\usepackage[T1]{fontenc}
\usepackage{lmodern}

\usepackage{amsmath, amsthm, amssymb}
\usepackage{xcolor}
\usepackage{microtype}
\usepackage{graphicx}
\usepackage{xspace}
\usepackage{setspace}
\usepackage{stmaryrd}
\usepackage{bbding}
\usepackage{tikz}
\usetikzlibrary{arrows}
\usepackage[margin=.6in]{geometry}
\usepackage{enumitem}

\usepackage{nicefrac}
	
\usepackage{lineno,hyperref}
\modulolinenumbers[5]

\journal{}

\theoremstyle{plain}
\newtheorem{thm}{Theorem}[section]
\newtheorem{lem}[thm]{Lemma}
\newtheorem{prop}[thm]{Proposition}

\newtheorem{corollary}[thm]{Corollary}
\newtheorem{question}[thm]{Question}

\theoremstyle{definition}
\newtheorem{defn}[thm]{Definition}

\theoremstyle{remark}
\newtheorem{remark}[thm]{Remark}

\newtheorem*{claim*}{Claim}

\newcommand{\A}{\mathcal{A}}

\newcommand{\K}{\mathrm{K}}
\newcommand{\KA}{\mathrm{KA}}

\renewcommand{\P}{\mathcal{P}}
\newcommand{\Q}{\mathcal{Q}}

\renewcommand{\S}{\mathcal{S}}

\newcommand{\U}{\mathcal{U}}

\newcommand{\X}{\mathcal{X}}

\newcommand{\MLR}{\mathrm{MLR}}

\newcommand{\dom}{\mathrm{dom}}

\newcommand{\Fin}{\mathrm{Fin}}

\definecolor{purple}{rgb}{.9,0.2,.9}

\newcommand{\cs}{2^\omega}
\newcommand{\uh}{{\upharpoonright}}
\renewcommand{\phi}{\varphi}
\newcommand{\str}{2^{<\omega}}
\newcommand{\binrat}{\mathbb{Q}_2}

\newcommand{\halts}{{\downarrow}}
\newcommand{\diverge}{{\uparrow}}

\newcommand{\T}{\mathrm{T}}

\newcommand{\cNCR}{\mathrm{NCR}_{\mathrm{comp}}}
\newcommand{\NCR}{\mathrm{NCR}}

\newcommand{\atoms}{\mathrm{Atoms}}

\newcommand{\llb}{\llbracket}
\newcommand{\rrb}{\rrbracket}

\renewcommand{\tt}{\mathrm{tt}}
\newcommand{\wtt}{\mathrm{wtt}}

\newcommand{\emptystr}{\varepsilon}

\newcommand{\claimqed}{\hfill$\Diamond$\smallskip}
\makeatletter
\newcommand{\claimqedmathmode}[1]{\ifmeasuring@#1\else\omit\hfill$\displaystyle\Diamond$\fi\ignorespaces}
\makeatother

\begin{document}
\sloppy

\begin{frontmatter}

\title{Randomness for computable measures and initial segment complexity}

\author{Rupert H\"olzl}
\address{Institute 1, Faculty of Computer Science\\
Universität der Bundeswehr München\\
Werner-Heisenberg-Weg 39\\
85577 Neubiberg\\
Germany}
\ead{r@hoelzl.fr}
\ead[url]{http://hoelzl.fr}

\author{Christopher P. Porter}
\address{Department of Mathematics and Computer Science\\
	Drake University\\
	Des Moines, IA 50311\\ 
USA}
\ead{cp@cpporter.com}
\ead[url]{http://cpporter.com}

\begin{abstract}
We study the possible growth rates of the Kolmogorov complexity of initial segments of sequences that are random with respect to some computable measure on $\cs$, the so-called proper sequences. Our main results are as follows:
(1)~We show that the initial segment complexity of a proper sequence $X$ is bounded from below by a computable function (that is, $X$ is complex) if and only if $X$ is random with respect to some computable, continuous measure. (2)~We prove that a uniform version of the previous result fails to hold:  there is a family of complex sequences that are random with respect to a single computable measure such that for every computable, continuous measure $\mu$, some sequence in this family fails to be random with respect to $\mu$.  (3) We show that there are proper sequences with extremely slow-growing initial segment complexity, that is, there is a proper sequence the initial segment complexity of which is infinitely often below every computable function, and even a proper sequence the initial segment complexity of which is dominated by all computable functions.  (4) We prove various facts about the Turing degrees of such sequences and show that they are useful in the study of certain classes of pathological measures on $\cs$, namely diminutive measures and trivial measures.

\end{abstract}


\end{frontmatter}

\section{Introduction}

The Levin-Schnorr Theorem establishes the equivalence of a certain measure-theoretic notion of typicality for infinite sequences (known as Martin-L\"of randomness) with a notion of incompressibility given in terms of Kolmogorov complexity.  Although the Levin-Schnorr Theorem is usually formulated for sequences that are random with respect to the Lebesgue measure on~$\cs$, it is well known that the theorem can be generalized to hold for any computable probability measure on~$\cs$.  More specifically, a sequence $X\in\cs$ is Martin-L\"of random with respect to a computable measure~$\mu$ if and only if the initial segment complexity of $X\uh n$ is bounded from below by $-\log\mu(X\uh n)$.  Thus we see that certain values of the measure~$\mu$ constrain the possible values of the initial segment complexities of the $\mu$-random sequences.

In this study, we further explore the interaction between computable measures and the initial segment complexity of the sequences that are random with respect to these measures (hereafter, we will refer to those sequences that are random with respect to a computable measure as \emph{proper} sequences, following the terminology of Zvonkin and Levin~\cite{ZvoLev70}).  In the first half of this article we focus on the relationship between a class of sequences known as \emph{complex sequences} and those sequences that are random with respect to a computable, continuous measure. 
  First studied systematically by Kjos-Hanssen et al.~\cite{KjoMerSte11} (but also studied earlier by Kanovi\v c~\cite{Kan70}), complex sequences are those sequences whose initial segment complexities are bounded below by some computable function.   We characterize the complex proper sequences as the sequences that are random with respect to some computable \emph{continuous} measure. This is done by studying the ``removability'' of $\mu$-atoms, that is, sequences~$X$ such that $\mu(\{X\})>0$: We show that if a sequence $X$ is complex and random with respect to some computable measure $\mu$, we can define a computable, continuous measure $\nu$ such that $X$ is random with respect to $\nu$ by removing the atoms from $\mu$ while preserving $X$'s randomness.  It is natural to ask whether this removal of atoms can always be carried out while preserving {\em all} non-atomic random sequences simultaneously, again assuming that all of these random sequences are complex. We show that this is not the case.
  
Using this characterization of complex sequences through computable continuous measures, we establish new results on the relationship between the notions of avoidability, hyperavoidability, semigenericity, and not being random for any computable, continuous measure.  More specifically, when restricted to the collection of proper sequences, we show that these four notions are equivalent to being complex.  We also study the granularity of a computable, continuous measure~$\mu$ and show that the inverse of the granularity function provides a uniform lower bound for the initial segment complexity of $\mu$-random sequences. 

In the second half of this article we turn our attention to atomic computable measures, that is, computable measures~$\mu$ that have $\mu$-atoms.  First, we study atomic measures $\mu$ with the property that every $\mu$-random sequence is either a $\mu$-atom or is complex.  We show that for such measures~$\mu$, even though the initial segment complexity of each non-atom $\mu$-random sequence is bounded from below by some computable function, there is in general no uniform computable lower bound for every non-atom $\mu$-random sequence.  Next, we construct a computable atomic measure $\mu$ with the property that the initial segment complexity of each $\mu$-random sequence dominates no computable function, and a computable atomic measure $\nu$ with the property that the initial segment complexity of each $\nu$-random sequence
is dominated by all computable functions.  The former sequences are called \emph{infinitely often anti-complex}, while the latter are known simply as \emph{anti-complex}.

Lastly, we study two specific kinds of atomic measures:  diminutive measures and trivial measures. Here, a measure $\mu$ is \emph{trivial} if $\mu(\atoms_\mu)=1$, and diminutive measures are defined as follows.  First, for $\mathcal{C}\subseteq\cs$, $\mathcal{C}$ is \emph{diminutive} if it does not contain a computably perfect subclass.  Let $\mu$ be a computable measure, and let $(\U_i)_{i\in\omega}$ be a universal $\mu$-Martin-L\"of test.  Then we say that $\mu$ is \emph{diminutive} if $\U_i^c$ is a diminutive $\Pi^0_1$ class for every $i$.
We show that while every computable trivial measure is diminutive, the converse does not hold. The proof of this last statement gives an alternative, priority-free proof of a result by Kautz~\cite{Kau91} showing that there is a computable, non-trivial measure $\mu$ such that there is no $\Delta^0_2$, non-computable~$X\in\MLR_\mu$.

The remainder of the article is organized as follows.  In Section \ref{sec:background}, we provide some background on computability theory and algorithmic randomness.  Section \ref{sec:basic-prop} contains a discussion of the basic properties of complex sequences.  The relationship between complex sequences and randomness with respect to computable, continuous measures is investigated in Section \ref{sec:complex-continuous}.  In Section \ref{sec:atomic} we study the behavior of complex proper sequences in the context of atomic measures.  Next, in Section \ref{sec-non-complex}, we consider non-complex proper sequences.  Lastly, in Section \ref{sec-triv-dim-measures} we relate the results of Section \ref{sec-non-complex} to the class of diminutive measures.

\section{Background}\label{sec:background}

\subsection{Some computability theory}  We assume that the reader is familiar with the basics of computability theory (for instance, the material covered in Soare~\cite[Ch.~I-IV]{Soa87}, Nies~\cite[Ch.~1]{Nie09}, or Downey and Hirschfeldt~\cite[Ch.~2]{DowHir10}).

For a Turing machine $N\colon \str\rightarrow\str$, let $\dom(N)$ be $\{\sigma \colon (\exists s)  N(\sigma)[s]\halts\}$, where $N(\sigma)[s]\halts$ means that there is some~$t\leq s$ such that if $N$ is given input $\sigma$, then it halts after $t$ steps.

We adopt the following conventions for Turing functionals on Cantor space:  A \emph{Turing functional} $\Phi\colon\subseteq\cs\rightarrow\cs$ is represented by a c.e.\ set $S_\Phi$ of pairs of strings $(\sigma,\tau)$ such that 
if $(\sigma,\tau),(\sigma',\tau')\in S_\Phi$ and $\sigma\preceq\sigma'$, then $\tau\preceq\tau'$ or $\tau'\preceq\tau$, where $\sigma\preceq\tau$ means that $\sigma$ is an initial segment of $\tau$. 
Moreover, for each $\sigma\in\str$, we define $\Phi^\sigma$ to be the maximal string (in the order given by $\preceq$) in 
$\{\tau\colon (\exists \sigma'\preceq\sigma)((\sigma',\tau)\in S_\Phi)\}$.  We then define $\Phi^X$ to be the maximal (in the order given by~$\preceq$) sequence~$z\in\str\cup\cs$ such that $\Phi^{X \uh n}$ is a prefix of~$z$ for all~$n$, and we set $\dom(\Phi)=\{X\in\cs\colon\Phi^X\in\cs\}$.  Lastly, for~$\tau \in \str$ let $\Phi^{-1}(\tau)$ be $\{ \sigma\in \str \colon \exists \tau' \succeq \tau \colon (\sigma,\tau')\in\Phi\}$. 

Recall that a Turing functional $\Phi$ is called a {\em tt-functional} if $\dom(\Phi)=\cs$, and $\Phi$ is called a {\em wtt-functional} if there exists a computable function $f\colon \omega \rightarrow \omega$ such that
for all $n$ and $\tau$ with $|\tau|=n$, if $\Phi^X\succeq \tau$ for some $X\in\cs$, then there is some $k\leq f(n)$ such that
$(X\uh k,\tau)\in S_\Phi$.

\subsection{Kolmogorov complexity and a priori complexity}

Recall that the {\em prefix-free Kolmogorov complexity} of a string $\sigma \in \str$ is defined as the length of the shortest program producing $\sigma$ by a fixed universal, prefix-free machine $U$, that is, as $\K(\sigma)=\min\{|\tau|\colon \tau \in \str \;\&\; U(\tau)=\sigma\}.$
We will more frequently work with a priori complexity, which is defined in terms of semi-measures on $\cs$.  Recall that if we denote the empty string by $\emptystr$ then
 a {\em semi-measure} is a function $\rho\colon\str\rightarrow[0,1]$ satisfying (i) $\rho(\emptystr) = 1$ and (ii) $\rho(\sigma)\geq\rho(\sigma0)+\rho(\sigma1)$ for all $\sigma \in \str$.  A semi-measure is {\em left-c.e.}\ if there is a computable function $\widetilde \rho \colon \str \times \omega \rightarrow \binrat$, non-decreasing in its second argument, and such that for all~$\sigma$, $\lim_{i \rightarrow +\infty} \widetilde \rho(\sigma,i) = \rho(\sigma)$.  That is, the values of $\rho$ on basic open sets are uniformly effectively approximable from below.

The following is a well-known fact first proved by Zvonkin and Levin~\cite{ZvoLev70}.
\begin{thm}
There exists a \emph{universal} left-c.e.\ semi-measure, that is, there exists a left-c.e.\ semi-measure $M$ such that, for every left-c.e.\ semi-measure~$\rho$, there exists some $c\in\omega$ such that $\rho\leq c\cdot M$. 
\end{thm}

For the rest of the article, fix a universal left-c.e.\ semi-measure $M$.  Then we define the \emph{a priori complexity} of $\sigma\in\str$, denoted $\KA(\sigma)$, to be $-\log M(\sigma)$. Notice that $\KA$ is monotonic with respect to the prefix order, that is, for $\sigma \prec \tau$ we have $\KA(\sigma) < \KA(\tau)$.

The following result provides a useful correspondence between Turing functionals and left-c.e.\ semi-measures.

\begin{thm}[Zvonkin and Levin~\cite{ZvoLev70}]\label{thm-MachinesInduceSemiMeasures}$ $
\begin{itemize}[noitemsep,topsep=0pt]
\item[(i)] For every Turing functional $\Phi$, the function defined by $\lambda_\Phi(\sigma)=\lambda(\llb\Phi^{-1}(\sigma)\rrb)
=\lambda(\{X\colon\Phi^X\succeq\sigma\})$ for every $\sigma\in\str$ is a left-c.e.\ semi-measure.
\item[(ii)] For every left-c.e.\ semi-measure $\rho$, there is a Turing functional $\Phi$ such that $\rho=\lambda_\Phi$. 
\end{itemize}
\end{thm}
For a left-c.e.\ semi-measure $\rho$, if $\Phi$ is a Turing functional such that $\lambda_\Phi=\rho$, we say that $\Phi$~\emph{induces}~$\rho$.  It is not hard to verify that if $\Phi$ is a universal Turing functional, that is, $\Phi(1^e0X)\simeq\Phi_e(X)$ for every $e\in\omega$ and $X\in\cs$, then $\Phi$ induces the universal left-c.e.\ semi-measure.

It is straightforward to show that $\KA(\sigma)\leq \K(\sigma)+O(1)$ for every ${\sigma\in\str}$.  
Another useful inequality comparing $\K$~and~$\KA$, not explicitly proved elsewhere, is the following.

\begin{lem}\label{lem-KA-K}
For $\sigma,\tau\in\str$, $\KA(\sigma\tau)\leq \K(\sigma)+\KA(\tau)$.
\end{lem}

\begin{proof}
See Appendix \ref{appendix}.
\end{proof}

\subsection{Computable measures on $\cs$}
For $\sigma\in\str$, we let $\llb\sigma\rrb=\{X\in\cs\colon \sigma\prec X\}$; for $V\subseteq\str$, we let $\llb V\rrb=\bigcup_{\sigma\in\str}\llb\sigma\rrb$.  A measure $\mu$ on $\cs$ is \emph{computable} if $\sigma \mapsto \mu(\llb\sigma\rrb)$ is computable as a real-valued function, that is, if there is a computable function $\widetilde \mu\colon\str\times\omega\rightarrow\binrat$ such that $|\mu(\llb\sigma\rrb)-\widetilde \mu(\sigma,i)|\leq 2^{-i}$ for every $\sigma\in\str$ and $i\in\omega$.  From now on, we will write $\mu(\llb\sigma\rrb)$ as $\mu(\sigma)$; and similarly, for~$V \subseteq \str$, we write $\mu(\llb V\rrb)$ as $\mu(V)$. The Lebesgue measure on $\cs$ is denoted $\lambda$.
For a  measure $\mu$ on $\cs$, $X\in\cs$ is an \emph{atom} of $\mu$ (or a \emph{$\mu$-atom}), denoted $X\in\atoms_\mu$, if $\mu(\{X\})>0$.  We call $\mu$ \emph{atomic} if $\atoms_\mu\neq\emptyset$; otherwise $\mu$ is \emph{continuous}.

\begin{prop}[Kautz~\cite{Kau91}]\label{prop-kautz}
If $\mu$ is a computable measure and $X\in\atoms_\mu$, then $X$ is computable.  
\end{prop}

If $\mu$ is a computable measure on $\cs$, then $\Phi$ is called {\em $\mu$-almost total} if $\mu(\dom(\Phi))=1$.  The correspondence given by Theorem \ref{thm-MachinesInduceSemiMeasures} can be extended to $\lambda$-almost total Turing functionals and computable measures.

\begin{samepage}
\begin{thm} \label{thm-measures-tf}\quad
\begin{itemize}[noitemsep,topsep=0pt]
\item[(i)] If $\mu$ is a computable measure and $\Phi$ is a $\mu$-almost total Turing functional, then the function $\mu_\Phi$ defined by ${\mu_\Phi(\X)=\mu(\Phi^{-1}(\X))}$ for all Borel $\X\subseteq\cs$ is a computable measure.
\item[(ii)] For every computable measure $\mu$, there is a $\lambda$-almost total Turing functional $\Phi$ such that $\mu=\lambda_\Phi$.
\end{itemize}
\end{thm}
\end{samepage}

\subsection{Martin-L\"of randomness} 
Let $\mu$ be a computable measure on $\cs$.  Recall that a \emph{$\mu$-Martin-L\"of test} is a sequence of uniformly $\Sigma^0_1$ classes $(\U_i)_{i\in\omega}$ such that $\mu(\U_i)\leq 2^{-i}$ for every $i\in\omega$. Moreover, $X\in\cs$ \emph{passes} a $\mu$-Martin-L\"of test $(\U_i)_{i\in\omega}$ if $X\notin\bigcap_{i\in\omega}\U_i$, and 
$X\in\cs$ is \emph{$\mu$-Martin-L\"of random}, denoted $X\in\MLR_\mu$, if $X$ passes every $\mu$-Martin-L\"of test.  We will write $\MLR_\lambda$ simply as $\MLR$.

\begin{defn}[Zvonkin and Levin~\cite{ZvoLev70}]
We call a sequence $X\in\cs$ {\em proper} if there exists a computable measure $\mu$ such that $X\in\MLR_\mu$. 
\end{defn}

As discussed in the introduction, the Levin-Schnorr Theorem plays a central role in this article.

\begin{thm}[Levin~\cite{Lev74}; Schnorr, see Chaitin~\cite{Cha75}]\label{thm-levin-schnorr}
$X\in\cs$ is Martin-L\"of random with respect to a computable measure~$\mu$ if and only if $
(\exists c)(\forall n)\;\K(X\uh n)\geq -\log(\mu(X\uh n))-c.$
\end{thm}

\noindent Levin~\cite{Lev73} also proved a version of Theorem~\ref{thm-levin-schnorr} with $\KA$ in place of $\K$. The following results will  be particularly useful.

\begin{thm}[Kurtz~\cite{Kur81}]\label{thm-kucera}
If $\mu$ is a computable measure and $\P$ is a $\Pi^0_1$ class such that $\mu(\P)=0$, then $\MLR_\mu\cap\P=\emptyset$.
\end{thm}

\begin{thm}[Randomness preservation]\label{thm-rand-pres}
Let $\mu$ be a computable measure and let $\Phi$ be a $\mu$-almost total Turing functional.  Then  $X\in\MLR_\mu$ implies $\Phi(X)\in\MLR_{\mu_\Phi}$.
\end{thm}

For a proof of this result, see, for instance, Bienvenu and Porter~\cite{BiePor12}.

\subsection{Complex sequences}\label{sec:basic-prop}

A function $f\colon\omega\rightarrow\omega$ is an \emph{order} if $f$ is unbounded and non-decreasing.  In the sequel, we will require that our orders~$g$ satisfy $g(0)=0$.  For an order $g$, we define $g^{-1}(n)=\min\{k\colon g(k)\geq n\}$. Notice that $g^{-1}$ is itself an order.

\begin{lem}\label{lem:gran0}
Let $f$ and $g$ be orders.
\begin{itemize}[noitemsep,topsep=0pt]
\item[(i)] If there is some $c\in\omega$ such that $g(n)\leq f(n)<g(n+c)$ for every $n$, then for every $k$, $g^{-1}(k)-c\leq f^{-1}(k)\leq g^{-1}(k)$.

\item[(ii)] If there is some $c\in\omega$ such that $f(n)\leq g(n)+c$ for every $n$, then $g^{-1}(k)\leq f^{-1}(k+c)$ for every $k$.
\end{itemize}
\end{lem}

\begin{proof}
See Appendix \ref{appendix}.
\end{proof}

As an immediate consequence of Lemma \ref{lem:gran0}(i), we have:

\begin{corollary}\label{cor:order-dom}
If $g$ is an order that dominates all computable orders, then $g^{-1}$ is an order that is dominated by all computable orders.
\end{corollary}

$X\in\cs$ is \emph{complex} if there is some computable order~$f\colon\omega\rightarrow\omega$ such that $\K(X\uh n)\geq f(n)$ for all $n$. Similarly $X\in\cs$ is \emph{autocomplex} if there is some $X$-computable order~$f\colon\omega\rightarrow\omega$ such that $\K(X\uh n)\geq f(n)$ for all $n$.  As shown by Kjos-Hanssen et al.~\cite{KjoMerSte11}, it is straightforward to show that $X\in\cs$ is complex (resp.\ autocomplex) if and only if there is some computable (resp.\ $X$-computable) order~$f$ such that $\K(X\uh f(n))\geq n$.

While we focus almost exclusively on complexity in this article, the following is worth noting:

\begin{prop}\label{prop-proper-autocomplex}
If $X\in\cs$ is proper and non-computable, then $X$ is autocomplex.
\end{prop}

\begin{proof}
Let $X\in\cs$ be a non-computable sequence, and suppose that $X\in\MLR_\mu$ for some computable measure $\mu$.  By Theorem \ref{thm-levin-schnorr}, we have $\K(X\uh n)\geq -\log(\mu(X\uh n))-c$ for some $c$.  Since the function $n\mapsto -\log(\mu(X\uh n))-c$ is $X$-computable, it follows that $X$ is autocomplex. 
\end{proof}

\begin{remark}\label{rmk-proper-converse}
The converse of Proposition \ref{prop-proper-autocomplex} does not hold: Miller~\cite{Mil11} constructed an~$X \in \cs$ of effective Hausdorff dimension $1/2$ which does not compute any sequence of higher effective Hausdorff dimension. Then $X$~is clearly complex (and thus autocomplex), but if $X$~computed any non-computable proper sequence it would compute a $Y \in \MLR$ by a result of Zvonkin and Levin~\cite{ZvoLev70} and Kautz~\cite[Theorem IV.3.14~(ii)]{Kau91}. Then $Y$ has effective Hausdorff dimension $1$, contradiction.	
\end{remark}

A variant of complexity that will feature prominently in this study is \emph{strong complexity}, which is defined in terms of a~priori complexity instead of prefix-free Kolmogorov complexity.

\begin{defn}
$X\in\cs$ is \emph{strongly complex} if there is some computable order~$f$ such that $\KA(X\uh n)\geq f(n)$ for all $n$.
\end{defn}

One particularly nice feature of strong complexity is given by the following result. 

\begin{lem}\label{lem-complex-inverse}
Let $X\in\cs$ and let $f$ and $g$ be (not necessarily computable) order.  
\begin{itemize}[noitemsep,topsep=0pt]
\item[(i)] If for an order~$h\colon\omega\rightarrow\omega$ we have  $\KA(X\uh g(n))\geq h(n)$ for every $n$, then $\KA(X\uh n) \geq h(g^{-1}(n)-1)$ for every $n$.
\item[(ii)] If $\KA(X\uh n)\geq f(n)$ for every $n$, then $\KA(X\uh f^{-1}(n))\geq n$ for every $n$.
\end{itemize}
\end{lem}

\begin{proof}
See Appendix \ref{appendix}.
\end{proof}

An immediate consequence of Lemma \ref{lem-complex-inverse} is an alternative characterization of strong complexity.

\begin{corollary}\label{cor-sc}
$X\in\cs$ is \emph{strongly complex} if and only if  there is some computable order~$f$ such that $\KA(X\uh f(n))\geq n$ for all $n$.
\end{corollary}

The following observation was already claimed by Higuchi et al.~\cite{HigHudSim14} without an explicit proof.
\begin{prop}\label{prop-complex-strcomplex}
$X\in\cs$ is complex if and only if $X$ is strongly complex.
\end{prop}

\begin{proof}
See Appendix \ref{appendix}.
\end{proof}

Lastly, we will make use of the following known result.

\begin{lem}[Bienvenu and Porter~\cite{BiePor12}]\label{lem-complex-wtt}
If $X$ is complex and $X\leq_{\mathrm{wtt}}Y$, then $Y$ is complex.
\end{lem}

\section{Complex sequences and computable, continuous measures}\label{sec:complex-continuous}

In this section, we study the relationship between proper complex sequences and computable, continuous measures; this can be seen as an extension of the work of Reimann~\cite{Rei08} who studied the relationship between complex sequences and (not necessarily computable) continuous  measures. First we show that for proper sequences being complex is equivalent to being random with respect to a computable, continuous measure; this result will be used as a technical tool throughout the rest of this article. Next we discuss some interesting consequences of this result. Finally we use the notion of granularity of a computable, continuous measure to show that the complexity of all sequences that are random for such a measure is witnessed by a common lower complexity bound.

\subsection{Characterizing complex proper sequences}\label{sdfhfdgjudstrefdgfhf}

According to Proposition \ref{prop-proper-autocomplex}, if $\mu$ is a computable measure and $X\in\MLR_\mu$ is non-computable, then $X$ is autocomplex.  However, in the case that $\mu$ is continuous, we can prove something stronger.  To do so, we will make use of the following lemma.

\begin{lem}\label{lem-tt-random}
Let $\mu$ be a computable, continuous measure and let $X\in\MLR_\mu$.  Then there is some $Y\in\MLR$ such that~$Y\leq_{\tt} X$.
\end{lem}

\begin{proof}
Bienvenu and Porter~\cite[Proposition 6.4]{BiePor12}, show that if $\nu$ is a computable, continuous measure, then there is some $\mathrm{tt}$-functional $\Phi$ such that $\MLR_{\nu_\Phi}=\MLR$.  Given such a $\Phi$ corresponding to the measure $\mu$, $\Phi(X)\in\MLR_{\nu_\Phi}$ by the Randomness Preservation Theorem~\ref{thm-rand-pres}.  Thus $X$~$\tt$-computes some $Y\in\MLR$. 
\end{proof}

\begin{thm}\label{thm-complex-continuous1}
If $X\in\cs$ is Martin-L\"of random with respect to a computable, continuous measure $\mu$, then $X$ is complex.
\end{thm}

\begin{proof}
If $X\in\cs$ is $\mu$-Martin-L\"of random for a continuous, computable measure $\mu$, then by Lemma \ref{lem-tt-random}, there is some Martin-L\"of random sequence $Y\leq_{\mathrm{tt}} X$.  By  Lemma \ref{lem-complex-wtt}, complexity is closed upwards under $\mathrm{wtt}$-reducibility (and hence $\mathrm{tt}$-reducibility), and since $Y$ is complex, so is $X$.
\end{proof}

We will give an alternative proof of this result in Section \ref{sec:granularity}, when we introduce the notion of the granularity of a computable measure.  

\begin{remark}\label{rmk-no-converse-complex}
The converse of Theorem \ref{thm-complex-continuous1} does not hold, since there are complex sequences that are not random with respect to \emph{any} computable measure.  For instance, the example cited in Remark~\ref{rmk-proper-converse} shows that the converse fails.
\end{remark}

We can, however, obtain a partial converse of Theorem \ref{thm-complex-continuous1} when we restrict to \textit{proper} complex sequences.  

\begin{thm}\label{thm-complex-continuous2}
If $X\in\cs$ is complex and proper, then there  is a computable, continuous measure $\nu$ such that $X\in\MLR_\nu$.
\end{thm}
Note that without the complexity assumption  this statement is trivially false by Theorem~\ref{thm-complex-continuous1}.

To prove Theorem~\ref{thm-complex-continuous2}, we need the following lemma, which will also be useful in the further course of the article.

\begin{lem}\label{lem:cont-measure-Pi01}
Let $\mu$ be a computable measure, let $X\in\MLR_\mu$, and let $\mathcal{P}$ be a $\Pi^0_1$ class containing $X$ but no computable members.  Then there is some computable, continuous measure $\nu$ such that $X\in\MLR_\nu$.
\end{lem}

\begin{proof}
If $\mu$ is continuous, then we are done.  Suppose, then, that $\mu$ is atomic.  By Theorem \ref{thm-kucera}, since $X\in\MLR_\mu$ and~$X\in\P$, it must be the case that $\mu(\P)>0$.  Let $T$ be a computable tree such that $\P=[T]$.  We define a computable, continuous measure $\nu$ such that $X\in\MLR_\nu$ in a straightforward fashion:  If $\sigma\in T$, we set $\nu(\sigma)=\mu(\sigma)$.  If $\sigma\notin T$, we have two cases to consider. Letting $\sigma^-$ denote the initial segment of $\sigma$ of length $|\sigma|-1$, we set $\nu(\sigma)=\mu(\sigma)$ if $\sigma^-\in T$.  However, if $\sigma^-\notin T$, then we set $\nu(\sigma)=\frac{1}{2}\nu(\sigma^-)$.

Clearly $\nu$ is a computable measure. To see that $\nu$ is continuous, given $Y\in\cs$, we have two cases to consider.  First, if $Y\in\P$, then since $\P$ contains no computable sequences, it follows from Proposition \ref{prop-kautz} that $Y$ cannot be a~$\nu$-atom.  Second, if $Y\notin\P$, then by definition of $\nu$, $\nu(Y\uh (n+1))=\frac{1}{2}\nu(Y\uh n)$ for all sufficiently large $n$, and hence $\nu(\{Y\})=\lim_{n\rightarrow\infty}\nu(Y\uh n)=0$.

Finally, we claim that $X\in\MLR_\nu$.  Since $X\in\MLR_\mu$, it follows from the Levin-Schnorr Theorem~\ref{thm-levin-schnorr} that ${\K(X\uh n)\geq -\log\mu(X\uh n)-O(1)}$ for every $n$.  But since $X\uh n\in T$ for every $n$, it follows that $\mu(X\uh n)=\nu(X\uh n)$ for every $n$.  Thus ${\K(X\uh n)\geq-\log\nu(X\uh n)-O(1)}$ for every $n$.  Again by the Levin-Schnorr Theorem, it follows that $X$ is $\nu$-Martin-L\"of random.
\end{proof}

\begin{proof}[Proof of Theorem \ref{thm-complex-continuous2}]
Let $X$ and $\mu$ be as in the statement of the theorem.  Since $X$ is complex, let $f$ be a computable order~$f$ such that $\K(X\uh n)\geq f(n)$ for all $n$.  The set $\P=\{Y\colon(\forall n)[\K(Y\uh n)\geq f(n)]\}$ is a $\Pi^0_1$~class that contains $X$ but does not contain any computable members.  Then, by Lemma \ref{lem:cont-measure-Pi01}, there is some computable, continuous measure $\nu$ such that $X\in\MLR_\nu$.
\end{proof}

\subsection{Consequences of the characterization}\label{shfjshrjhsjf}

There are a number of interesting results that follow from Theorems \ref{thm-complex-continuous1} and \ref{thm-complex-continuous2}.  The first involves the class $\cNCR$, which is the collection of sequences that are not random with respect to any computable, continuous measure, first introduced by Bienvenu and Porter~\cite{BiePor12}.  $\cNCR$ can be seen as the computable analogue of the collection $\NCR_1$ of sequences that are not Martin-L\"of random with respect to \emph{any} continuous measure, which has been studied by Reimann and Slaman~\cite{ReiSla07, ReiSla08}.  

It follows from a result of Reimann~\cite[Theorem 5]{Rei08} that no complex sequence is contained in $\NCR_1$.  This does not characterize $\NCR_1$, as there are continuum many non-complex sequences (which follows, for instance, from Theorem \ref{thm-ioc-ioac} below) but only countably many sequences in $\NCR_1$ as shown by Reimann and Slaman~\cite{ReiSla07}.  However, when we restrict to proper sequences, we get a precise characterization of $\cNCR$ in terms of complexity.

\begin{corollary}\label{cor-ncr}
Let $X\in\cs$ be proper.  Then $X\notin\mathrm{NCR}_{\mathrm{comp}}$ if and only if $X$ is complex.
\end{corollary}

\begin{proof}
The direction ($\Rightarrow$) follows from Theorem \ref{thm-complex-continuous1}, while the direction ($\Leftarrow$) follows from Theorem \ref{thm-complex-continuous2}.
\end{proof}

Another consequence of Theorems \ref{thm-complex-continuous1} and \ref{thm-complex-continuous2} involves the notion of semigenericity, which was studied by Demuth and later by Demuth and Ku\v cera~\cite{Dem87, DemKuc87}.

\begin{defn}
A non-computable sequence $X\in\cs$ is \emph{semigeneric} if every $\Pi^0_1$ class containing~$X$ also contains a computable sequence.
\end{defn}

\begin{thm}\label{thm-ncr-semigeneric}
Let $X\in\cs$ be proper and non-computable.  Then $X\in\cNCR$ if and only if $X$ is semigeneric.
\end{thm}

\begin{proof}
Suppose that $X$ is not computable and that $X\notin\cNCR$.  Then there is a computable, continuous measure $\mu$ such that $X\in\MLR_\mu$.  If $(\U_i)_{i\in\omega}$ is a universal $\mu$-Martin-L\"of test, then there is some $i$ such that $X\in\U_i^c$.  But $\U_i^c$ is a $\Pi^0_1$~class containing no computable sequence, since $\mu$ has no atoms.  Thus $X$ is not semigeneric.

Suppose now that $X$ is not semigeneric.  Then there is some $\Pi^0_1$ class with no computable members that contains $X$.  But since $X\in\MLR_\mu$, we can apply Lemma \ref{lem:cont-measure-Pi01} to conclude that $X\in\MLR_\nu$ for some computable, continuous measure~$\nu$.  Thus $X\notin\cNCR$.
\end{proof}

From Corollary \ref{cor-ncr} and Theorem \ref{thm-ncr-semigeneric} we can conclude the following:

\begin{corollary}\label{cor-complex-semigeneric}
Let $X\in\cs$ be proper.  Then one and only one of the following holds:
\begin{itemize}[noitemsep,topsep=0pt]
\item[(i)] $X$ is complex; 
\item[(ii)] $X$ is semigeneric.
\end{itemize}
\end{corollary}

\begin{remark}\label{rmk-complex-semigeneric}
One can directly prove that complex sequences are not semigeneric without the assumption of properness: As in the proof of Theorem~\ref{thm-complex-continuous2}, for a given  sequence $X\in\cs$ whose complexity is witnessed by a computable order~$f$, the $\Pi^0_1$ class $\P=\{Y\colon(\forall n)[\K(Y\uh n)\geq f(n)]\}$ contains $X$ but no computable sequences.
\end{remark}

One surprising, significant consequence of Corollary \ref{cor-complex-semigeneric} concerns the notions of avoidability and hyperavoidability, systematically studied by Miller.  

\begin{defn}[Miller~\cite{Mil02}]  Let $X\in\cs$.
\begin{itemize}[noitemsep,topsep=0pt]
\item[(i)] Then $X$ is \emph{avoidable} if there is some partial computable function $p$ such that for every computable set $M$ and every index~$e$ for $M$, $p(e)\halts$ and $X\uh p(e)\neq M\uh p(e)$.
\item[(ii)] Moreover, $X$ is \emph{hyperavoidable} if $X$ is avoidable with a total $p$ as above.
\item[(iii)] $X$ is \emph{conditionally avoidable} if $X$ is avoidable but not hyperavoidable.
\end{itemize}
\end{defn}

Miller proved a strong separation of the notions of avoidability and hyperavoidability.

\begin{thm}[Miller~\cite{Mil02}]\label{thm-cond-av}
There are continuum many conditionally avoidable sequences.
\end{thm}

The following two results show the relevance of avoidability and hyperavoidability to the present discussion.

\begin{thm}[Demuth and Ku\v cera~\cite{DemKuc87}; see also Ku{\v{c}}era, Nies, and Porter~\cite{MR3411165}]\label{thm-demuth-avoidable}
Let $X\in\cs$ be a non-computable sequence.  Then $X$ is avoidable if and only if $X$~is not semigeneric.
\end{thm}

\begin{thm}[Kjos-Hanssen et al.~\cite{KjoMerSte11}]\label{thm-hyperavoidable-complex}
Let $X\in\cs$.  Then $X$ is hyperavoidable if and only if $X$ is complex.
\end{thm}

Combined with Corollary \ref{cor-complex-semigeneric}, Theorems \ref{thm-demuth-avoidable} and \ref{thm-hyperavoidable-complex} immediately yield the following.
\begin{samepage}
\begin{corollary}
Let $X$ be proper and non-computable.  The the following are equivalent.
\begin{itemize}[noitemsep,topsep=0pt]
\item[(i)] $X$ is complex.
\item[(ii)] $X\notin\cNCR$.
\item[(iii)] $X$ is not semigeneric.
\item[(iv)] $X$ is avoidable.
\item[(v)] $X$ is hyperavoidable.
\end{itemize}
In particular, no conditionally avoidable sequence is proper.
\end{corollary}
\end{samepage}
The relationship between these various concepts is summed up in Figure \ref{fig:diagram}.


\begin{figure}[tb]
\begin{center}
\begin{tikzpicture}[scale=.6]
  \node[minimum width = 3.2cm,minimum height=0.7cm] (notNCRcomp) at (0,0) {$X\not\in \cNCR$};

  \node[minimum width = 3.2cm,minimum height=0.7cm] (complex) at (10,0) {$X$ complex};
  \node[minimum width = 3.2cm,minimum height=0.7cm] (hyperavoid) at (10,-4) {$X$ hyperavoidable};

  \node[minimum width = 3.2cm,minimum height=0.7cm] (notsemigen) at (20,0) {$X$ not semigeneric};
  \node[minimum width = 3.2cm,minimum height=0.7cm] (avoid) at (20,-4) {$X$ avoidable};

  \draw [-implies,double equal sign distance] (notNCRcomp) to node [midway,above] {\scriptsize Theorem \ref{thm-complex-continuous1}} (complex);
  \draw [-implies,double equal sign distance] (complex) to node [midway,above] {\scriptsize Remark \ref{rmk-complex-semigeneric}} (notsemigen);
  \draw [-implies,double equal sign distance] (hyperavoid) to node [midway,above] {\scriptsize by definition} (avoid);
  \draw [-implies,double equal sign distance,dashed]  (notsemigen) to  [out=140,in=40] node [midway,above] {\scriptsize Theorem \ref{thm-ncr-semigeneric} }(notNCRcomp) ;
  
  \draw [implies-implies,double equal sign distance] (complex) to node [midway,right] {\scriptsize Theorem \ref{thm-hyperavoidable-complex}} (hyperavoid);
  \draw [implies-implies,double equal sign distance] (notsemigen) to node [midway,right] {\scriptsize Theorem \ref{thm-demuth-avoidable} } (avoid);
  
  \draw [dotted,thick]  (5,4.5) to  (5,3.4);
  \draw [dotted,thick]  (5,2.8) to  (5,0.7);
  \draw [dotted,thick]  (5,-0.3) to  (5,-5.5) node[below]{\scriptsize Separated by Remark \ref{rmk-no-converse-complex}};  
  
  \draw [dotted,thick]  (15,4.5) to  (15,3.4);
  \draw [dotted,thick]  (15,2.8) to  (15,0.8);
  \draw [dotted,thick]  (15,-0.3) to  (15,-3.2);
  \draw [dotted,thick]  (15,-4.3) to  (15,-5.5) node[below]{\scriptsize Separated by Theorem \ref{thm-cond-av}};
\end{tikzpicture}
\end{center}

\caption{The dashed implication is only true for non-computable proper sequences $X$. Without this assumption we obtain the dotted separations. All other implications are unconditionally true.}
\label{fig:diagram}
\end{figure}
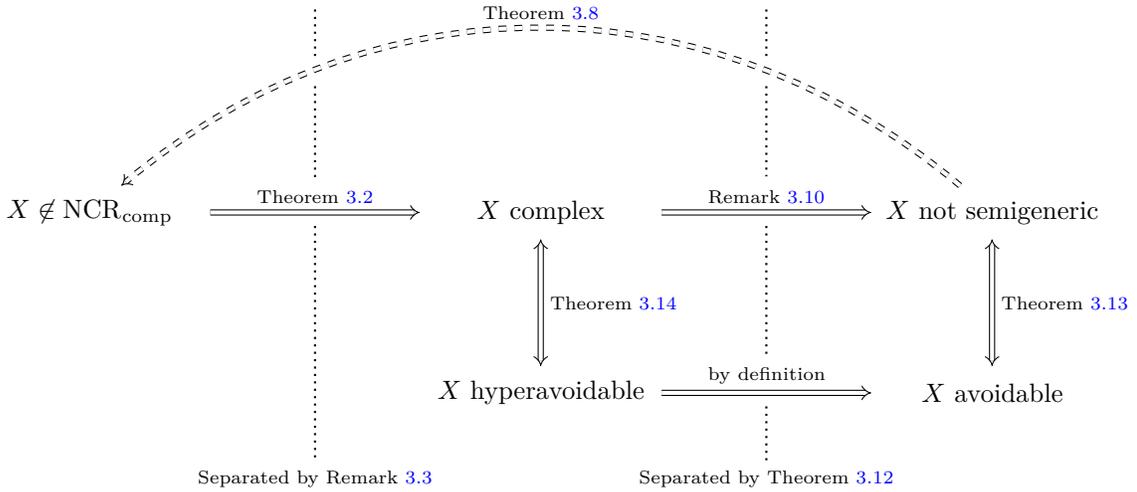

\subsection{The granularity of a continuous, computable measure}\label{sec:granularity}

\begin{defn}[Barmpalias, Greenberg, Montalbán, and Slaman~\cite{BarGreMon11}]
Let $\mu$ be a continuous measure.  The \emph{granularity function of $\mu$}, denoted~$g_\mu$, is the order mapping $n$ to the least $\ell$ such that $\mu(\sigma)<2^{-n}$ for every $\sigma$ of length $\ell$.
\end{defn}

By the following result, $g_\mu$ need not be computable in general.
\begin{prop}
There is a computable, continuous measure $\mu$ such that the granularity $g_\mu$ of $\mu$ is not computable.
\end{prop}

\begin{proof} We provide a sketch of the proof and leave the details to the reader.  The idea is to define $g_\mu(n)$ to be either~$2n$ or~$2n+1$ depending on whether $\phi_n(n)\halts$ or $\phi_n(n)\diverge$.  To do so, for all strings $\sigma$ of length $2n$ except for the two rightmost strings of length $2n$ (denoted $\tau_0$ and $\tau_1$), $\mu(\sigma)$ is  defined to be $\frac{1}{2}\mu(\sigma)$.  For $\tau_0$ and $\tau_1$, we set $\mu_s(\tau_0)=0$ and $\mu_s(\tau_1)=\mu_s(\tau_1^-)$ for every stages $s$ such that $\phi_{n,s}(n)\diverge$; if there is some least stage $t$ such that $\phi_{n,t}(n)\halts$, we set $\mu_t(\tau_0)=2^{-t}$ and $\mu_t(\tau_1)=\mu_t(\tau_1^-)-2^{-t}$.  One can easily verify that the resulting function $g_\mu$ has the desired properties.
\end{proof}

One useful property of the granularity of a measure $\mu$ is that its inverse  $g_\mu^{-1}$ provides a (not necessarily computable) lower bound on the initial segment complexity of the $\mu$-random sequences.

\begin{prop}\label{prop:gran}
Let $\mu$ be a computable, continuous measure and let $X\in\MLR_\mu$.  Then there is some~$c$ such that for every $n$,
$\KA(X\uh g_\mu(n))\geq n-c$ and thus $\KA(X\uh n)\geq g^{-1}_\mu(n)-c$.
\end{prop}

\begin{proof}
See Appendix \ref{appendix}.
\end{proof}

Despite the fact that $g^{-1}_\mu$ need not be computable, there exists a computable function very close to it.

\begin{lem}\label{lem:gran1}
For every computable, continuous measure $\mu$ on $\cs$, there is a computable order~$f\colon\omega\rightarrow\omega$ such that
${g_\mu(n)\leq f(n)< g_\mu(n+2)}$ for every $n$.
\end{lem}

\begin{proof}
See Appendix \ref{appendix}.
\end{proof}

To obtain the desired computable lower bound for the initial segment complexity of the $\mu$\nobreakdash-random sequences for a computable, continuous measure $\mu$, we apply  Lemma \ref{lem:gran0}(i) to the inequality $g_\mu(n)\leq f(n)< g_\mu(n+2)$ from Lemma \ref{lem:gran1} to conclude that $g_\mu^{-1}(n)-2\leq f^{-1}(n)\leq g_\mu^{-1}(n)$.  Combined with Proposition \ref{prop:gran}, this yields the following.

\begin{thm}\label{thm-global-order}
If $\mu$ is a continuous, computable measure, there is a computable order~$h$ such that $|h(n)-g^{-1}_\mu(n)|\leq O(1)$  and for every $X\in\MLR_\mu$, $\K(X\uh n)\geq h(n)$.
\end{thm}

\section{Complex proper sequences and atomic measures}\label{sec:atomic}

According to Theorem \ref{thm-complex-continuous2}, if a sequence $X$ is complex and random with respect to some computable measure $\mu$, we can define a computable, continuous measure $\nu$ such that $X$ is random with respect to $\nu$ by removing the atoms from $\mu$; $X$'s~randomness could be preserved in this process by using that $\mu$-atoms are recognizably non-complex, which allowed distinguishing them from~$X$. It is natural to ask whether this removal of atoms can always be carried out while preserving {\em all} non-atomic random sequences simultaneously, again assuming that all of these random sequences are complex. We give a negative answer to this question with the following theorem.

\begin{thm}\label{thm:uniform-atoms2}
There is a computable, atomic measure $\mu$ such that 
\begin{itemize}[noitemsep,topsep=0pt]
\item[(i)] every $X\in\MLR_\mu\setminus\atoms_\mu$ is complex, and
\item[(ii)] there is no computable, continuous measure $\nu$ with $\MLR_\mu\setminus\atoms_\mu\subseteq\MLR_\nu.$
\end{itemize}
\end{thm}

Recall that in Lemma~\ref{lem:gran1} we showed that the granularity function of every computable, continuous measure is dominated by some computable function. Clearly, atomic measures do not have a granularity function, but as a useful tool for proving the above theorem we can define a notion of local granularity, that is, a function that behaves like the granularity of a measure $\mu$ along some a fixed non-$\mu$-atom $X$. 

\begin{defn}  Let a computable measure $\mu$ and $X\notin\atoms_\mu$ be given. The \emph{local granularity function} $g_\mu^X$ of $\mu$ along $X$ is defined to be $g_\mu^X(n) = \min\{ k     \colon  \mu(X\uh k) < 2^{-n}        \}$.
\end{defn}

A key result about the local granularity function of a measure $\mu$ used in the proof of Theorem \ref{thm:uniform-atoms2}  is the following.

\begin{lem}\label{lem-granularity-bound}
Let $\mu$ be a computable measure and $f$ be a computable order.  
Then there is a constant $d$ such that, for all~$X\notin\atoms_\mu$ with 
$\KA(X\uh f(n))\geq n$ for every $n$, we have that 
$f(n+d)\geq g_\mu^X(n)$ for all $n$.
\end{lem}

\begin{proof}
Since $M$ is a universal semi-measure, there is some constant $e$ such that for all $\sigma \in \str$, $\mu(\sigma)\leq 2^{e} \cdot M(\sigma)$. By the definition of a priori complexity and our assumptions on $X$, we have that for all $n$,
\[
-\log M(X\uh f(n+1))=\KA(X\uh f(n+1))\geq n+1>n,
\]
which implies that $M(X\uh f(n+1))< 2^{-n}$.  Thus 
\[
\mu(X\uh f(n+1))\leq 2^e\cdot M(X\uh f(n+1))<2^{-(n-e)}
\]
and so for every $n$ we have $\mu(X\uh f(n+e+1))<2^{-n}$.  This implies that $g_\mu^X(n)\leq f(n+e+1)$ for every $n$.
\end{proof}

\begin{proof}[Proof of Theorem \ref{thm:uniform-atoms2}]
Let $(\phi_i)_{i\in\omega}$ be an effective enumeration of all partial computable functions.   The key to our construction is to define the measure $\mu$ so that for every $i\in\mathrm{Tot}=\{i\colon\phi_i\text{ is total}\}$, there is some $X\in\MLR_\mu\setminus\atoms_\mu$ such that $g_\mu^X(n)$ is not dominated by $\phi_i$.  We adopt the standard convention that for all $i$ and~$n$, if $\phi_i(n)$ halts in $s$~steps, then $\phi_i(n)\leq s$. We further assume that for $n\geq 0$, the number of stages it takes for~$\phi_i(n+1)$ to halt is always strictly greater than the number of stages it takes for $\phi_i(n)$ to halt.   Lastly, we assume that if $\phi_i(n)\halts$, then $\phi_i(k)\halts$ for all $k<n$.

For each $i\in\omega$ we will define a sequence of natural numbers $(n_i[s])_{s\in\omega}$  such that at stage $s$ we will evaluate the function~$\phi_i$ at $n_i[s]+i$. For each $i\in\omega$, we will also define a sequence of finite sets of strings $(L_i[s])_{s\in\omega}$ which correspond to action taken at stage $s$ to ensure that $\phi_i$ does not bound $g_\mu^X(n)$ for the measure $\mu$ that we are defining and for some specific~$X\in\cs$ to be identified shortly.  More precisely, for each $\tau\in L_i[s]$, we will define $\mu$ on extensions of $\tau$ of the form~$\tau 0^j$ while we are waiting for the computation of $\phi_i(n_i[s]+i)$ to converge (and during this time $L_i[s]$ will not vary as~$s$~grows).  If this specific computation converges at some stage $t$, for every $\tau\in L_i[t-1]$, we will place two incompatible extensions of $\tau 0^k$ into $L_i[t]$, where $k$ encodes some information about the amount of time needed for the computation in question to converge.  

\medskip

\noindent {\em Construction.} At stage $0$, for each $i\in\omega$, we set $L_i[0]=\{0^i1\}$, $n_i[0]=0$, and $\mu(\varepsilon)=1$.

\noindent At stages $s > 0$, supposing that $\mu$ is defined on all strings of length $<s$, we define $\mu$ for strings $\sigma$ of length $s$ as follows:
\begin{itemize}
\item[(i)] If $\sigma=0^s$ or $\sigma=0^{s-1}1$ then we set $\mu(\sigma)=2^{-s}$. 
\item[(ii)] If $\sigma$ is of the form $\tau0^j$ for some $\tau \in L_i[s-1]$ and some $j\in\omega$ and some $i < s$,
then we check whether $\phi_i(n_i[s-1]+i)[s]\!\downarrow$ and $\phi_i(n_i[s-1]+i)[s-1]\!\uparrow$ (that is, $s$ is the least stage such that $\phi_i(n_i[s-1]+i)[s]\halts$). 
\begin{itemize}
\item If not, then we set $\mu(\tau0^j)=\mu(\tau0^{j-1})$ and $\mu(\tau0^{j-1}1)=0$, $L_i[s]=L_i[s-1]$, and $n_i[s]=n_i[s-1]$.

\item If yes, then we set $\mu(\tau0^j)=\mu(\tau0^{j-1}1)=\frac{1}{2}\mu(\tau0^{j-1})$, we 
let $L_i[s] = L_i[s-1] \cup \{\tau0^j, \tau0^{j-1}1\}$, and we set $n_i[s]=n_i[s-1]+1$.
\end{itemize}
\item[(iii)] For all other $\sigma\in 2^{s}$, let $\mu(\sigma)=0$.
\end{itemize}
Lastly, for all $i\geq s$, we let $L_i[s]=L_i[s-1]$.
This finishes the construction.

\medskip

\noindent {\em Verification.} We now verify a series of claims.

\smallskip

\noindent {\em Claim 1.}  If $i\notin\mathrm{Tot}$, then $\MLR_\mu\cap\llb 0^i1\rrb\subseteq\atoms_\mu$.

\smallskip

\noindent {\em Proof of Claim 1.}
If $\phi_i$ is not total, then let $n_i=\lim_{s\rightarrow\infty}n_i[s]$, so that $n_i$ is the least $n$ such that $\phi_i(n+i)\diverge$.   Moreover, let $L_i=\lim_{s\rightarrow\infty}L_i[s]$, which exists by the construction.
Then inside the neighborhood $\llb0^i1\rrb$, $\mu$ is concentrated on a finite number of elements:
\begin{itemize}
\item if $n_i=n_i[0]$, then $\mu$ is concentrated on $\{0^i 10^\omega\}$;
\item otherwise $n_i=n_i[s]$ and $L_i=L_i[s]$, where $s$ is the first stage such that $\phi_i(n_i+i-1)[s]\!\downarrow$.  In this case, $\mu$ is concentrated on sequences of the form $\tau0^\omega$ where $\tau \in L_i$.\end{itemize}
In either case, $\MLR_\mu \cap \llb0^i1\rrb$ consists entirely of $\mu$-atoms.\claimqed

\noindent {\em Claim 2.}  If $i\in\mathrm{Tot}$, then $\llb 0^i1\rrb\cap\atoms_\mu=\emptyset$.

\smallskip

\noindent {\em Proof of Claim 2.}    Assume otherwise and let $X\in \llb0^i1\rrb$ be a $\mu$-atom. By our construction, along any $Y\in\cs$, either $\mu(Y\uh n)\rightarrow 0$ or $\mu(Y\uh n)$ is eventually constant and equal to $2^{-m}$ for some $m$ (as $\mu$ only assigns measures of the form $2^{-m}$ to basic open sets).  Thus, since $X$ is a $\mu$-atom, there is some $m$ such that $\mu(\{X\})=2^{-m}$.  Note also that if $Y\in \llb 0^i1\rrb$ has the property that $\mu(Y\uh n)> 0$ for every $n\in\omega$, it follows that for every $s$, there is some $k_s$ such that $Y\uh k_s\in L_i[s]$.  

Let $X\uh \ell$ be the shortest initial segment of $X$ such that $\mu(X\uh \ell)=2^{-m}$. Then by construction $X\uh \ell$ is added to $L_i[s]$ at some stage $s$. But as $\phi_i$ is total, we know that $\phi_i(n_i[s]+1+i)$ will eventually terminate at some least stage $t$.  At this stage~$t$, there is some $k_t$ such that  $X\uh k_t\in L_i[t]$.  By the construction, at this stage we set $\mu(X\uh k_t)=\frac{1}{2}\mu(X\uh \ell)=2^{-(m+1)}$, which contradicts the assumption that $\mu(\{X\})=2^{-m}$.
\claimqed

\noindent {\em Claim 3.}  If $i\in\mathrm{Tot}$, then for all $X\in\MLR_\mu\cap\llb 0^i1\rrb$, $X$ is complex.

\smallskip

\noindent {\em Proof of Claim 3.}
To see this, first note that every $X\in \llb0^i1\rrb$ such that $\mu(X\uh n)>0$ for every~$n$ can be written as
\[
X=0^i1\,0^{t_0}b_0\, 0^{t_1}b_1\,\dotsc\,0^{t_j}b_j\,\dotsc,
\]
where $b_i\in\{0,1\}$ for each $i$ and the sequence $(t_\ell)_{\ell\in\omega}$ is defined inductively as follows.  First, $t_0$ is the least stage $s$ such that~$\phi_i(i)[s]\halts$.  Having defined $t_0,\dotsc,t_k$, we define $t_{k+1}$ to be the least stage~$s\geq 0$ such that~$\phi_i((k+1)+i)$ halts in~$s+\sum_{j=0}^k t_j$~steps (such an $s$ exists by our convention that the number of stages needed to compute $\phi_i(n+1)$ is strictly greater than the number of stages needed to compute $\phi_i(n)$ for every $i$ and $n$).  Note that the values $(t_\ell)_{\ell\in\omega}$ only depend on the index~$i$ and not on~$X \in \llb0^i1\rrb$.  Moreover, since $\phi_i$ is total, the sequence $(t_\ell)_{\ell\in\omega}$ is computable.

Now let $\Phi$ be the total Turing functional such that for $Y=y_0y_1y_2\dotsc$ 
with $y_i\in \{0,1\}$ for all~$i$
we have
\[
\Phi(Y)=0^i1\,0^{t_0}y_0\,0^{t_1}y_1\dotsc\,0^{t_j}y_j\,\dotsc,
\]
where the $t_\ell$'s are as above.  Note that $\Phi$ is injective, and thus  induces a continuous measure~$\lambda_\Phi$ such that 
$\lambda_\Phi=2^{i+1}\mu\uh_{\llb 0^i1\rrb}$. It is not hard to show that this latter fact implies that $\MLR_{\lambda_\Phi}=\MLR_\mu\cap\llb0^i1\rrb$. 
By Claim 2, it follows that $\lambda_\Phi$ is continuous. But then by Theorem \ref{thm-complex-continuous1}, every $X\in \MLR_{\lambda_\Phi}=\MLR_\mu\cap\llb0^i1\rrb$ is complex, which establishes the claim.
\claimqed

\noindent {\em Claim 4.}  If $i\in\mathrm{Tot}$, then for all $X\in\MLR_\mu\cap\llb 0^i1\rrb$, $g^X_\mu$ is not dominated by $\phi_i$.

\smallskip

\noindent {\em Proof of Claim 4.}
Suppose that $i\in\mathrm{Tot}$ and let $X\in\MLR_\mu\cap\llb0^i1\rrb$ be given.  Then by definition of~$\mu$, $X$ must have the form
\[
X=0^i1\,0^{t_0}b_0\, 0^{t_1}b_1\,\dotsc\,0^{t_j}b_j\,\dotsc,
\]
where the sequence $(t_\ell)_{\ell\in\omega}$ is as above and $b_i\in \{0,1\}$ for every $i\in\omega$. Now given $k\geq 0$, if $n$ is least such that $\mu(X\uh n)<2^{-(i+k)}$, then by our construction we have
\[
X\uh n=0^i1\,0^{t_0}b_0\, 0^{t_1}b_1\,\dotsc\,0^{t_k}b_k.
\]
Recall that  by the definition of the sequence $(t_\ell)_{\ell\in\omega}$, for each $k\geq 0$, $\phi_i(k+i)$ halts in $\sum_{j=0}^{k} t_j$ steps and thus ${\phi_i(k+i)\leq\sum_{j=0}^{k} t_j}$.
It follows from this and the definition of the local granularity~$g_\mu^X$ that
\[
g_\mu^X(i+k)=n>\sum_{j=0}^{k} t_j\geq \phi_i(i+k) .  
\]
Thus for all $k\geq 0$, we have $g_\mu^X(i+k)>\phi_i(i+k)$ and thus $g_\mu^X$ dominates $\phi_i$.
\claimqed

\noindent {\em Claim 5.} There is no computable order~$h$ such that for every $X\in\MLR_\mu\setminus\atoms_\mu$ we have $\KA(X\uh n)\geq h(n)$.

\medskip

\noindent {\em Proof of Claim 5.}
Suppose for the sake of contradiction there is a computable order~$h$ such that for every ${X\in\MLR_\mu\setminus\atoms_\mu}$ we have $\KA(X\uh n)\geq h(n)$ for all $n$. Hence by Lemma \ref{lem-complex-inverse}(ii) we have $\KA(X\uh h^{-1}(n))\geq n$ for every $n$ and every $X\in\MLR_\mu\setminus\atoms_\mu$. Applying Lemma \ref{lem-granularity-bound} to~$h^{-1}$, it follows that there is some $d$ such that for every $X\in\MLR_\mu\setminus\atoms_\mu$ and every $n$ we have $g_\mu^X(n)\leq h^{-1}(n+d)$. As $h^{-1}$ is computable, there is some~$i$ such that $\phi_i$ computes the function~$n \mapsto h^{-1}(n+d)$. But by Claim 4, for every $X\in \MLR_\mu\cap\llb0^i1\rrb$, $g_\mu^X$ is not dominated by~$\phi_i$, contradiction.  
\claimqed

\medskip

We now conclude the proof of Theorem~\ref{thm:uniform-atoms2}:  Suppose there is a computable, continuous measure~$\nu$ such that ${\MLR_\mu\setminus\atoms_\mu\subseteq\MLR_\nu}$.  By Theorem \ref{thm-global-order}, there is a computable order~$h$ such that for every ${X\in\MLR_\nu}$ we have $\KA(X\uh n)\geq h(n)$. In particular this holds for every ${X \in \MLR_\mu\setminus\atoms_\mu}$, which contradicts Claim 5.  
\end{proof}

\section{Non-complex proper sequences}\label{sec-non-complex}

In this section, we study the initial segment complexity of sequences that are random with respect to a computable atomic measure.  We begin by reviewing several notions of non-complexity defined in terms of prefix-free complexity and then prove equivalent formulations in terms of a priori complexity.

\subsection{Notions of non-complexity}\label{subsec-noncomplex}

In this subsection we recall several notions of non-complex\-ity.  The first notion, known as \emph{infinitely often complexity}, was introduced by Hölzl and Merkle~\cite{HolMer10} and further studied by Higuchi and Kihara~\cite{HigKih14}.  The second notion, known as \emph{anti-complexity}, was introduced in Franklin et al.~\cite{FraGreSte13}.  Lastly, we introduce \emph{infinitely often anti-complexity}, which is a natural modification of anti-complexity.  Although the former two notions are typically expressed in terms of plain Kolmogorov complexity, they can be equivalently expressed in prefix-free complexity, which we do here.

\begin{samepage}

\begin{defn} Let $X\in\cs$.
\begin{itemize}[noitemsep,topsep=0pt]
\item[(i)] $X$ is \emph{infinitely often complex} (or \emph{i.o.\ complex}) if there is some computable order~$f$ such that $\K(X\uh f(n))\geq n$ for infinitely many $n$.
\item[(ii)] $X$ is \emph{anti-complex} if for every computable order~$f$ we have $\K(X\uh f(n))\leq n$ for almost every~$n$.
\item[(iii)] $X$ is \emph{infinitely often anti-complex} (or \emph{i.o.\ anti-complex}) if for every computable order~$f$ we have $\K(X\uh f(n))\leq n$ for infinitely many~$n$.
\end{itemize}
\end{defn}

\end{samepage}

It is clear that a sequence is anti-complex if and only if it is not i.o.\ complex, and that a sequence is  i.o.\ anti-complex if and only if it is not complex.  Thus, every sequence is either complex, i.o.\ complex but not complex (and hence i.o.\ anti-complex), or anti-complex.  As we will discuss in Section \ref{sec-non-complex}, there are non-computable proper sequences belonging to each of the latter two classes.

Note that it is not equivalent to define a sequence $X$ to be i.o.\ complex if there is some computable order~$f$ such that $\K(X\uh n)\geq f(n)$ for infinitely many $n$.  For as observed by Hölzl and Merkle~\cite{HolMer10}, every sequence satisfies this latter property.  This follows from the fact that (i) for any $X\in\cs$, $\K(X\uh n)\geq \K(n)-O(1)$ and that (ii) there are computable functions $f$ such that $f(n)\leq \K(n)$ for infinitely many $n$ (for example, Solovay functions; see Bienvenu and Downey~\cite{BieDow09}).  

Similarly, it is not true that $X$ is anti-complex if and only if for every computable order~$f$\!, ${\K(X\uh n)\leq f(n)}$ for almost every $n$.  This follows from the fact established by Bienvenu and Downey~\cite[Theorem 4.3]{BieDow09} that (i)~there is a computable order~$g$ such that $\K{(X\uh n)\leq g(n)+O(1)}$ if and only if $X$ is $\K$-trivial, that is, $\K(X\uh n)\leq \K(n)+O(1)$, and that (ii)~not every anti-complex sequence is $\K$-trivial (Franklin et al.~\cite{FraGreSte13} show that every high degree contains an anti-complex sequence, but every $\K$-trivial sequence has low Turing degree).

By contrast, given that a sequence is i.o.\ anti-complex if and only if it is not complex, it is straightforward to show that $X\in\cs$~is i.o.\ anti-complex if and only if for every computable order~$f$ we have $\K(X\uh n))\leq f(n)$ for infinitely many~$n$.

Just as with complexity, we can equivalently formulate the three notions of non-complexity in terms of a priori complexity.  

\begin{prop}\label{prop-non-complex-KA} Let $X\in\cs$.
\begin{itemize}[noitemsep,topsep=0pt]
\item[(i)] $X$ is i.o.\ complex if and only if there is a computable order~$f$ such that $\KA(X\uh f(n))\geq n$ for infinitely many $n$.
\item[(ii)]  $X$ is anti-complex if and only if for every computable order~$f$, $\KA(X\uh f(n))\leq n$ for almost every $n$. 
\item[(iii)] $X$ is i.o.\ anti-complex if and only if for every computable order~$f$, $\KA(X\uh f(n))\leq n$ for infinitely many $n$.

\end{itemize}
\end{prop}

\begin{proof}
See Appendix \ref{appendix}.
\end{proof}

We also have the following additional characterizations of the above notions of non-complexity.

\begin{prop} \label{prop-non-complex-KA-2} Let $X\in\cs$.
\begin{itemize}[noitemsep,topsep=0pt]
\item[(i)] $X$ is i.o.\ complex if and only if there is a computable order~$f$ such that $\KA(X\uh n)\geq f(n)$ for infinitely many~$n$.
\item[(ii)] $X$ is anti-complex if and only if for every computable order~$f$, we have ${\KA(X\uh n)\leq f(n)}$ for almost every~$n$.
\item[(iii)]  $X$ is i.o.\ anti-complex if and only if for every computable order~$f$, we have ${\KA(X\uh n)\leq f(n)}$ for infinitely many~$n$.
\end{itemize}
\end{prop}

\begin{proof}
See Appendix \ref{appendix}.
\end{proof}

Note that if we replace $\KA$ with $\K$ in Proposition \ref{prop-non-complex-KA-2}, the resulting statements are not true by the discussion prior to Proposition \ref{prop-non-complex-KA}.  We now study the relationship between the above notions and proper sequences.

\subsection{Anti-complexity and randomness}  We first discuss anti-complex, proper sequences.  Franklin et al.~\cite{FraGreSte13} showed that every high degree contains an anti-complex sequence.  The converse fails, as an anti-complex sequence need not compute a fast-growing function.  In particular, there is an anti-complex sequence of hyperimmune-free degree.
 However, the situation markedly differs when we restrict 
to proper sequences.  First, we show that every high, random $\wtt$-degree contains an anti-complex proper sequence, and then we prove that every anti-complex proper sequence has high Turing degree.  Recall that $X\in\cs$ has high Turing degree if~$X'\geq\emptyset''$, or equivalently, if $X$ computes a function that dominates every computable function.

\begin{thm}\label{thm:ac-wtt}
Suppose that $X\in\MLR$ and $f\leq_{\mathrm{wtt}}X$ dominates all computable functions.  Then there is
an anti-complex proper sequence $Y\equiv_\T X$.
\end{thm}

\begin{proof}
Suppose that $\phi^X=f$, and let $g(Z,n)$ bound the running time of $\phi^Z(n)$ for $Z\in\cs$. For $Z\in\cs$ write $Z=z_0z_1z_2\dotsc$ with $z_i\in \{0,1\}$ for all $i$. We inductively define a functional $\Gamma$ on input $Z\in\cs$ in terms of blocks of output $\gamma_i^Z$ as follows (hereafter we will write each $\gamma_i^Z$ as $\gamma_i$).
For each $i\geq 0$, $\gamma_i$ will either be finite in length or undefined.  For~$n\geq 0$, $\gamma_{n}$ is determined inductively as follows:
\[
\gamma_{n}=
	\left\{
		\begin{array}{ll}
			1^{g(Z,n)}\,0\,z_n & \mbox{if } \phi^Z(n)\halts, \\
			\mbox{undefined} & \mbox{if } \phi^Z(k)\diverge \mbox{ for some $k \leq n$}.
		\end{array}
	\right.
\]
We then define $\Gamma(Z)$ by
\[
\Gamma(Z)=
	\left\{
		\begin{array}{ll}
			\gamma_0\,\gamma_1\,\gamma_2\dotsc & \mbox{if } |\gamma_i|<\infty \mbox{ for every } i\geq 0, \\
			\gamma_0\,\gamma_1\,\gamma_2\dotsc\gamma_{i-1}1^\omega & \mbox{if $i$ is least such that } \gamma_i \mbox{ is undefined.}
		\end{array}
	\right.
\]

If $\phi^Z$ is not total, then there is a least $k$ such that $\phi^Z(k)\diverge$, and 
\[
\Gamma(Z)=1^{g(Z,0)}\,0\,z_0\;\dotsc 1^{g(Z,k-1)}\,0\,z_{k-1}\;1^\omega.
\]
If on the other hand $\phi^Z$ is total, then 
\[
\Gamma(Z)=1^{g(Z,0)}\,0\,z_0\,1^{g(Z,1)}\,0\,z_1\dotsc
\]

Define $\ell(Z,n)=2(n+1)+\sum_{i=0}^ng(Z,i)$, and set $Y=\Gamma(X)$ and $\ell(n)=\ell(X,n)$. If $h$ is the computable function that bounds the use of the computation $\phi^X=f$ (where without loss of generality we can assume that $h$ is an order), then since we can compute each of the values $f(0),\dotsc, f(n)$ from $X\uh h(n)$, it follows that we can compute $Y\uh \ell(n)$ from $X\uh h(n)$ via $\Gamma$.  That is, $(X\uh h(n), Y\uh\ell(n))\in S_\Gamma$, where $S_\Gamma$ is the c.e.\ set of pairs of strings that generates $\Gamma$.  This implies that $\lambda_\Gamma(Y\uh\ell(n))\geq 2^{-h(n)}$.  Since $2^c\cdot M\geq \lambda_\Gamma$ for some $c\in\cs$ by the universality of $M$ among all left-c.e.\ semi-measures, it follows that $2^c\cdot M(Y\uh\ell(n))\geq \lambda_\Gamma(Y\uh\ell(n))\geq 2^{-h(n)}$.  Taking the negative logarithm of both sides yields $\KA(Y\uh\ell(n))\leq h(n)+c$, and hence
$\ell(\ell^{-1}(n))\geq n$ implies, by the monotonicity of~$\KA$, that
\[
\KA(Y\uh n)\leq \KA(Y\uh \ell(\ell^{-1}(n)))\leq h(\ell^{-1}(n))+c.
\]
We claim that $Y$ is anti-complex.  By the characterization of anti-complexity given by Proposition~\ref{prop-non-complex-KA-2}(ii), it suffices to show that $h(\ell^{-1}(n))+c$ is dominated by all computable functions.
First observe that since $f$ (and hence $g$) dominates all computable functions, it follows that $\ell$ dominates all computable functions.  Since $\ell$ is non-decreasing and unbounded, $\ell^{-1}$~is well-defined and is dominated by all computable functions by Corollary \ref{cor:order-dom}.  Given some computable order~$k$, suppose that there are infinitely many $n$ such that
${h(\ell^{-1}(n))\geq k(n)-c}$.
Since $h^{-1}$ is non-decreasing, it follows that there are infinitely many $n$ such that
\begin{equation}\label{eq-ac1}
h^{-1}(h(\ell^{-1}(n))-1)\geq h^{-1}(k(n)-c-1).
\end{equation}
Since $h^{-1}(h(\ell^{-1}(n))-1)$ is the least $j$ such that $h(j)\geq h(\ell^{-1}(n))-1$, it follows that
\begin{equation}\label{eq-ac2}
\ell^{-1}(n)\geq h^{-1}(h(\ell^{-1}(n))-1).
\end{equation}
Combining inequalities (\ref{eq-ac1}) and (\ref{eq-ac2}), it follows that there are infinitely many $n$ such that
\[
\ell^{-1}(n)\geq  h^{-1}(k(n)-c-1).
\]
But this contradicts the fact that $\ell^{-1}$ is dominated by all computable functions.  It follows that $\KA(Y\uh n)\leq k(n)$
for almost every $n$.  Since $k$ was an arbitrary computable order, it follows by Proposition \ref{prop-non-complex-KA-2}(ii) that $Y$ is anti-complex.
\end{proof}

\begin{thm}\label{thm-ac-high}
Let $X\in\cs$ be non-computable, anti-complex, and proper.  Then $X$ is high.
\end{thm}

\begin{proof}
Let $\mu$ be a computable measure such that $X\in\MLR_\mu$.  We show that $X$ computes some function $f$ that dominates all computable functions.  Since $X$ is non-computable it is not an atom of $\mu$, and therefore, by the Levin-Schnorr Theorem~\ref{thm-levin-schnorr}, there is some $X$-computable order~$p(n)$ such that
$\KA(X\uh n)\geq p(n)$
for all $n$.  By Lemma \ref{lem-complex-inverse}(ii), $\KA(X\uh p^{-1}(n))\geq n$ for all $n$, and setting $f(n)=p^{-1}(2n)$, we have $\KA(X\uh f(n))\geq 2n$ for all $n$.  To see that $f(n)=p^{-1}(2n)$ is the desired function, let $g$ be an arbitrary computable order.  Then for almost every $n$,
$\KA(X\uh g(n))\leq n<2n\leq \KA(X\uh f(n))$.
By the monotonicity of $\KA$, $g(n)\leq f(n)$ for almost every $n$.  Since $g$ was arbitrary, it follows that $X$ is high.
\end{proof}

Note that Theorem \ref{thm:ac-wtt} is a partial converse of Theorem \ref{thm-ac-high}.  It is an open question whether the full converse holds.

\begin{question}
If $X\in\MLR$ is high, is there some anti-complex proper $Y\equiv_\T X$?
\end{question}

\subsection{I.o.\ anti-complexity and randomness}

We now turn to i.o.\ anti-complex proper sequences. 
Their existence follows as a corollary from the following result of Bienvenu and Porter~\cite{BiePor12}. Recall that $\cNCR$ is the collection of sequences that are not random with respect to any computable, continuous measure. 

\begin{thm}[Bienvenu and Porter~\cite{BiePor12}]\label{thm-bphi}
Let $\mathbf{a}$ be a random Turing degree.  Then $\mathbf{a}$ contains some proper~${X\in\cNCR}$ if and only if $\mathbf{a}$ is hyperimmune.
\end{thm}

By Theorem \ref{thm-complex-continuous2}, if $X\in\cNCR$ is proper, then $X$ cannot be complex.  Then by Theorem~\ref{thm-bphi}, we can conclude:

\begin{corollary}\label{cor-hi}
Let $\mathbf{a}$ be a random Turing degree.  Then $\mathbf{a}$ contains an i.o.\ anti-complex proper sequence if and only if $\mathbf{a}$~is~hyperimmune.
\end{corollary}

Note that the proof of Theorem \ref{thm-ac-high} can be used to show that any $X$ that is i.o.\ anti-complex and proper must be of hyperimmune degree, thereby
yielding a proof of the left-to-right direction of Corollary \ref{cor-hi}.

One can directly verify that the sequences constructed in the proof of Theorem \ref{thm-bphi} are i.o\@.~anti-complex.  Moreover, in the case that the degree $\mathbf{a}$ is not high, Theorem \ref{thm-ac-high} guarantees that the i.o\@.~anti-complex proper sequence in $\mathbf{a}$ cannot be anti-complex; that is, such a sequence must be both i.o\@.~anti-complex and i.o\@.~complex.  However, in the case that $\mathbf{a}$ is high, we have more work to do to show that it contains a proper sequence that is both i.o.\ anti-complex and i.o.\ complex.  We will show this by  modifying the proof of Theorem \ref{thm-bphi}.  The main idea is to define a proper sequence that contains infinitely many long blocks of 1's and infinitely many long blocks of unbiased random bits.  Our proof applies to all proper sequences of hyperimmune degree, not just proper sequences of high degree.

\begin{thm}\label{thm-ioc-ioac}
Every hyperimmune random Turing degree contains a proper sequence that is both i.o\@.~complex and i.o\@.~anti-complex.
\end{thm}

\begin{proof}
Given a hyperimmune random Turing degree $\mathbf{a}$, let $X\in\mathbf{a}$ be a Martin-L\"of random sequence and let~$\phi^X$~be an~$X$\nobreakdash-computable function that is not dominated by any computable function.  We will assume that for any $Z\in\cs$, if~$\phi^Z(n)\halts$ in $s$ stages then for all $i<n$, $\phi^Z(i)\halts$ in less than~$s$ stages.  Let $f(Z,n)$ bound the running time of $\phi^Z(n)$; that~is, $f(Z,n)$~is the least stage $s$ such that~$\phi^Z_s(n)\halts$; if no such stage exists, then we set $f(Z,n)=+\infty$.  
Again using the convention that for all~$n$, if $\phi^Z(n)$~halts in $s$ steps, then $\phi(n)\leq s$, it follows that if $\phi^Z$ is total, then $f(Z,\cdot)$ dominates $\phi^Z$.  We will additionally adopt the convention that if $\phi^Z(n)$~halts in $s$ steps, then $n\leq s$.  Note that by our conventions, if $\phi^Z$ is total, then the function $f(Z,\cdot)$ is strictly increasing and $f(Z,n)\geq n$ for every $n$.

Given $f$ as above, we simultaneously define in a recursive way a total Turing functional $\Xi$ on input~$Z\in\cs$, as well as a function $g\colon\cs\times\omega\rightarrow\omega$ and a sequence of strings $(\tau_i)_{i\in\omega}$ that are dependent upon $Z$.
We first define $g$ and the sequence~$(\tau_i)_{i\in\omega}$ as follows.

\begin{itemize}
\item[(i)] For $k=0$, if $f(Z,0)<+\infty$, then set $g(Z,0)=f(Z,0)$ and ${\tau_0=Z\uh (f(Z,0)+1)}$; if ${f(0)=+\infty}$, then set ${g(Z,0)=+\infty}$ and $\tau_0$ is undefined. 
\item[(ii)] For $k=n+1$, if $f(Z,i)<+\infty$ for each $i\leq n+1$, then set $g(Z,n+1)=f(Z,|\tau_n|)$ and let $\tau_{n+1}$~be the block $Z(|\tau_n|)\cdots Z(|\tau_n|+j_{n+1})$, where \[j_{n+1}=\sum_{i=0}^{n+1}2^{(n+1)-i}(g(Z,i)+1);\] 
if $f(Z,i)=+\infty$ for $i\leq n+1$, then set $g(Z,n+1)=+\infty$ and let $\tau_{n+1}$ be undefined.  
\end{itemize}
Note that if $\tau_{n+1}$ is defined, then $|\tau_{n+1}|=j_{n+1}$.  

Next, as in the proof of Theorem \ref{thm:ac-wtt}, we will describe the output $\Xi(Z)$ in terms of blocks of output~$\xi_i^Z$ (which, as in the proof of Theorem \ref{thm:ac-wtt}, will be written as $\xi_i$). 
 For each $i\geq 0$, $\xi_i$ will either be finite in length or undefined.  For $n\geq 0$, $\xi_{n}$~is determined inductively as follows:
\[
\xi_{n}=
	\left\{
		\begin{array}{ll}
			1^{g(Z,n)}\,0\,\tau_n & \mbox{if } \phi^Z(n)\halts, \\
			\mbox{undefined} & \mbox{if } \phi^Z(k)\diverge \mbox{ for some $k\leq n$}.
		\end{array}
	\right.
\]
We then define $\Xi(Z)$ by
\[
\Xi(Z)=
	\left\{
		\begin{array}{ll}
			\xi_0\,\xi_1\,\xi_2\dotsc & \mbox{if } |\xi_i|<\infty \mbox{ for every } i\geq 0, \\
			\xi_0\,\xi_1\,\xi_2\dotsc\xi_{i-1}\,1^\omega & \mbox{if $i$ is least such that } \xi_i \mbox{ is undefined.}
		\end{array}
	\right.
\]

\noindent If $\phi^Z$ is not total, then there is a least $k$ such that $\phi^Z(k)\diverge$, and it is easy to see that in this case
\[
\Xi(Z)=1^{g(Z,0)}\,0\,\tau_0\dotsc 1^{g(Z,k-1)}\,0\,\tau_{k-1}\,1^\omega.
\]
If on the other hand $\phi^Z$ is total, then 
\[
\Xi(Z)=1^{g(Z,0)}\,0\,\tau_0\,1^{g(Z,1)}\,0\,\tau_1\dotsc;
\]
in particular, $\phi^Z$~does not end in a tail of 1's.

To clarify the definitions of $g$ and $(\tau_i)_{i\in\omega}$, we establish the following claim, which states that if $\phi^Z$~is total, then infinitely often at least half of the
bits of $\Xi(Z)$ will consist of runs of bits from~$Z$.

\smallskip

\noindent {\em Claim.}   Suppose that $\phi^Z(i)\halts$ for every $i\leq k$. Let \[\sigma_k=1^{g(Z,0)}\,0\,\tau_0\,1^{g(Z,1)}\,0\,\tau_1\,\dotsc\tau_{k-1}\,1^{g(Z,k)}\,0,\] where $\tau_0\,\tau_1\,\dotsc\,\tau_k\preceq Z$.  Then $|\tau_k|=|\sigma_k|$.

\smallskip

\noindent {\em Proof of Claim.} We prove this by induction.  For $k=0$, $\sigma_0=1^{g(Z,0)}\,0$, hence $|\sigma_0|=g(Z,0)+1=f(Z,0)+1=|\tau_0|$ by definition.  Now suppose that $|\tau_n|=|\sigma_n|$.  
Then 

\begin{align*}
|\sigma_{n+1}|&=|\sigma_n|+|\tau_n|+g(Z,n+1)+1\\
&=2|\tau_n|+g(Z,n+1)+1 & (\text{by induction hypothesis})\\
&=2\sum_{i=0}^{n}2^{n-i}(g(Z,i)+1)+g(Z,n+1)+1 & (\text{by definition of $\tau_n$})\\
&=\sum_{i=0}^{n}2^{(n+1)-i}(g(Z,i)+1)+g(Z,n+1)+1\\
&=\sum_{i=0}^{n+1}2^{(n+1)-i}(g(Z,i)+1)=j_{n+1}=|\tau_{n+1}|,
\end{align*}
as needed.\claimqed

To see that $\Xi$ is total we have two cases to consider.  First, if $\phi^Z$ is a total function, then by the definition of $\Xi$, it is clear that $\Xi(Z)\in\cs$. However, 
if $\phi^Z$ is not a total function, then $\Xi(Z)=\sigma 1^\omega$ for some $\sigma\in\str$, since we will have $g(Z,n)=+\infty$ for some $n$.  

Now let $Y=\Xi(X)$ (where $X\in\MLR$ is the sequence of hyperimmune degree introduced above).  One can readily verify that $g(X,\cdot)$ dominates $f(X,\cdot)$, and hence it follows that  $g(X,\cdot)$ is not dominated by any computable function.  Since $\Xi$ is total, it thus follows by preservation of randomness (Theorem~\ref{thm-rand-pres}) that $Y$ is random with respect to the measure induced by $\Xi$, and hence $Y$ is proper.

To prove that $Y$ is i.o.\ anti-complex, we modify the proof of Bienvenu and Porter~\cite[Proposition 5.9]{BiePor12} to show that $Y$ is not complex.  To simplify our notation, we will write $g(X,n)$ simply as $g(n)$.  By Lemma \ref{lem-KA-K} and the fact that $\KA(1^n)=O(1)$ for every $n$, we have
\begin{equation}\label{eq-ioac1}
\begin{array}{rl}
&\KA(1^{g(0)}\,0\,\tau_0\dotsc1^{g(n)}\,0\,\tau_n\,1^{g(n+1)})\\
\leq& \K(1^{g(0)}\,0\,\tau_0\dotsc1^{g(n)}\,0\,\tau_n)+\KA(1^{g(n+1)})+O(1)\\
\leq& \K(1^{g(0)}\,0\,\tau_0\dotsc1^{g(n)}\,0\,\tau_n)+O(1).
\end{array}
\end{equation}
Note that $g(n)\geq |\tau_i|$ for $i<n$.  Indeed, by the definition of $g$ and the conventions on $f$ laid out above, we have 
\[
g(n)=f(Z,|\tau_n|)> f(Z,|\tau_i|)\geq |\tau_i|
\]
for $i<n$. From this and the fact that $g$ is strictly increasing (since $f$ is strictly increasing), it follows that 
\begin{equation}\label{eq-ioac2}
|1^{g(0)}\,0\,\tau_0\dotsc1^{g(n)}\,0\,\tau_n|\leq (2n+1)g(n)+|\tau_n|+(n+1).
\end{equation}
But since $|\tau_n|=|\sigma_n|=|1^{g(0)}\,0\,\tau_0\dotsc1^{g(n)}\,0|$ by the the above Claim and $g(n)\geq n$, it follows from Equation~(\ref{eq-ioac2}) that 
\begin{equation}\label{eq-ioac3}
\begin{split}
|1^{g(0)}\,0\,\tau_0\dotsc1^{g(n)}\,0\,\tau_n|&\leq (4n+2)g(n)+2(n+1)\\
&\leq (4g(n)+2)g(n)+2(g(n)+1)\\
&=4g(n)(g(n)+1)+O(1).
\end{split}
\end{equation}
Using the fact that $\K(\sigma)\leq|\sigma|+2\log(|\sigma|)+O(1)$,  we can conclude from Equation (\ref{eq-ioac3}) that
\begin{equation}\label{eq-ioac4}
\begin{array}{rl}
&\K(1^{g(0)}\,0\,\tau_0\dotsc1^{g(n)}\,0\,\tau_n)\\
\leq& 4g(n)(g(n)+1)+2\log(4g(n)(g(n)+1))+O(1)\\
\leq& 4(g(n))^2+4g(n)+4\log(g(n))+O(1).
\end{array}
\end{equation}
Let $\ell(n)=4n^2+4n+4\log(n)+c$, where $c$ is a constant that is larger than the sum of the additive  constants in Equations~(\ref{eq-ioac1})~and~(\ref{eq-ioac4}).  Let $h$ be a computable order.  We apply the following lemma, which is a straightforward modification of a result in Bienvenu and Porter~\cite[Lemma 5.8]{BiePor12}.

\begin{lem}\label{lem-orders}
If $g$ is a strictly increasing function that is not dominated by any computable function and $h$~and~$\ell$~are computable orders, then for infinitely many $n$,
$\ell(g(n)) < h(g(n + 1))$.
\end{lem}
Combining Equations (\ref{eq-ioac1}) and (\ref{eq-ioac4}) and applying Lemma \ref{lem-orders} yields, for infinitely many $n$,
\[
\KA(Y\uh g(n+1))\leq \KA(1^{g(0)}\,0\,\tau_0\dotsc1^{g(n)}\,0\,\tau_n\,1^{g(n+1)})\leq \ell(g(n))\leq h(g(n+1)),
\]
where the first inequality is by the monotonicity of a priori complexity. Since $h$ was an arbitrarily chosen computable order, it follows that $Y$ is i.o.\ anti-complex.

Finally we show that $Y$ is i.o.\ complex.  If we set $Y\uh k_n=1^{g(0)}\,0\,\tau_0\,1^{g(1)}\,0\,\tau_1\dotsc\,1^{g(n)}\,0\,\tau_n$ and $X\uh j_n=\tau_0\,\tau_1\,\tau_2\dotsc\tau_n$, then it is clear that one can effectively recover $X\uh j_n$ from $Y\uh k_n$, since each $\tau_i$ has the same length as the string of bits that precedes it in $Y$.  Thus, $\K(Y\uh k_n)\geq \K(X\uh j_n)-O(1)$. For every $n$, applying the above Claim $n$~times, we obtain~$j_n\geq (1/2) k_n$.  Thus, since $X$ is Martin-L\"of random, by the Levin-Schnorr Theorem~\ref{thm-levin-schnorr},
\[
\K(Y\uh k_n)\geq \K(X\uh j_n)\geq j_n-O(1)\geq(1/2)k_n-O(1).
\]
Thus, there exist infinitely many $n$ such that $\K(Y\uh n)\geq (1/2)n-O(1)$ and hence $X$ is i.o.\ complex.  Lastly, $Y\in\mathbf{a}$ since clearly $Y\geq_\T X$.
\end{proof}

\section{Diminutive measures}\label{sec-triv-dim-measures}

In this final section we consider a class of computable measures called \emph{diminutive} measures.  As we will see, diminutive measures are relevant to the study of the initial segment complexity of proper sequences.  Diminutive measures are defined in terms of diminutive~$\Pi^0_1$~classes, introduced by Binns~\cite{Bin08}.

\subsection{Computably perfect and diminutive $\Pi^0_1$ classes}

For a $\Pi^0_1$ class $\mathcal{P}$, recall that the set of {\em extendible nodes} of $\mathcal{P}$ is $\text{Ext}(\mathcal{P})=\{\sigma\in\str\colon (\exists X\in\mathcal{P})[X\succ \sigma]\}$.  As defined in the introduction, $\mathcal{C}\subseteq\cs$ is \emph{computably perfect} if there is a computable, strictly increasing function $f$ such that for each $n\in\omega$ and each $\sigma\in\text{Ext}(\mathcal{P})$ such that $|\sigma|=f(n)$, there exist incomparable $\tau_1,\tau_2\in\text{Ext}(\mathcal{P})$ extending $\sigma$ such that $|\tau_1|=|\tau_2|=f(n+1)$. Moreover, $\mathcal{C}$ is \emph{diminutive} if it does not contain a computably perfect subclass. 

We will restrict our attention to diminutive $\Pi^0_1$ classes. Binns gave two useful characterizations of these classes in terms of complex sequences and $\mathrm{wtt}$-covers, where a collection $\A\subseteq\cs$ is a \emph{$\mathit{wtt}$-cover for $\cs$} if every $X\in\cs$ is $\mathrm{wtt}$-reducible to some~$Y\in\A$.

\begin{thm}[Binns~\cite{Bin08}]\label{thm-binns}
For a $\Pi^0_1$ class $\mathcal{P}\subseteq\cs$, the following are equivalent:
\begin{enumerate}[noitemsep,topsep=0pt]
\item $\mathcal{P}$ contains a computably perfect $\Pi^0_1$ class containing $X$;
\item $\mathcal{P}$ contains a complex element $X$;
\item $\mathcal{P}$ contains a $\mathrm{wtt}$-cover for $\cs$.
\end{enumerate}
\end{thm}

As a brief aside, we prove a hitherto unnoticed consequence of Theorem \ref{thm-binns}  and Lemma~\ref{lem-complex-wtt}.

\begin{prop}
Let $\mu$ be a computable measure.  If there is a complex $X\in\MLR_\mu$, then there are continuum many complex $X\in\MLR_\mu$.
\end{prop}

\begin{proof}
If $(\U_i)_{i\in\omega}$ is a universal $\mu$-Martin-L\"of test, then since $X\in\MLR_\mu$ there is some $i$ such that~$X\in\U_i^c$.  Since $\U_i^c$~is a $\Pi^0_1$ class and $X$ is a complex member of $\U_i^c$, by Theorem \ref{thm-binns}, $\U_i^c$~contains a $\mathrm{wtt}$-cover for $\cs$.  In particular, for every complex $Y\in\cs$, there is some $Z\in\U_i^c$ such that $Y\leq_{\mathrm{wtt}}Z$.  But by Lemma \ref{lem-complex-wtt}, being complex is closed upwards under~$\leq_{\mathrm{wtt}}$, and since $Y$ is complex, so is $Z$.  Thus $\U_i^c$ contains continuum many complex sequences.  But since $\U_i^c$ consists entirely of $\mu$-Martin-L\"of random sequences, it follows that $\MLR_\mu$ contains continuum many complex sequences.
\end{proof}

From the definition of diminutive classes we can derive a natural definition of diminutive measures.

\begin{defn}[Porter~\cite{Por12}]
Let $\mu$ be a computable measure, and let $(\U_i)_{i\in\omega}$ be a universal $\mu$-Martin-L\"of test.  Then we say that $\mu$ is \emph{diminutive} if $\U_i^c$ is a diminutive $\Pi^0_1$ class for every $i$.
\end{defn}

\begin{prop}\label{prop-diminutive}
A computable measure $\mu$ is diminutive if and only if there is no complex $X\in\MLR_\mu$.
\end{prop}

\begin{proof}
Suppose there is some complex $X\in\MLR_\mu$.  Then if $(\U_i)_{i\in\omega}$ is a universal $\mu$-Martin-L\"of test, there is some $i$ such that $X\in\U_i^c$.  But by Theorem \ref{thm-binns}, this implies that $\U_i^c$ contains a computably perfect $\Pi^0_1$ subclass containing $X$, and thus $\U_i^c$ is not diminutive.  For the other direction, if there is no complex $X\in\MLR_\mu$, then by Theorem \ref{thm-binns}, each $\U_i^c$ contains no complex sequence and is thus diminutive.
\end{proof}

Note that it is not the case that every computable atomic measure is diminutive.  For instance, the computable measure constructed in the proof of Theorem \ref{thm:uniform-atoms2} is atomic but not diminutive.  

By combining Proposition~\ref{prop-diminutive} with Theorem~\ref{thm-complex-continuous1}, it is clear that 
if $X$ is random with respect to some computable diminutive measure, then it cannot be random with respect to any computable continuous measure. In light of this observation it is reasonable to state the following open question.
\begin{question}
	Let $X$ be proper and not random with respect to  any computable diminutive measure. Is~$X$ necessarily random with respect to a computable continuous measure? Equivalently by Theorems~\ref{thm-complex-continuous1} and~\ref{thm-complex-continuous2}, is $X$ necessarily complex?
\end{question}

\subsection{Trivial measures}

One subcollection of the computable, diminutive measures are the computable, trivial measures, studied systematically by Porter~\cite{Por15}.

\begin{defn}
A measure $\mu$ is \emph{trivial} if $\mu(\atoms_\mu)=1$.
\end{defn}

The following answers an open question of Porter~\cite{Por12}.

\begin{prop}
Every computable trivial measure is diminutive.
\end{prop}

\begin{proof}
Let $\mu$ be a computable trivial measure.  Suppose $\mu$ is not diminutive.  Then by Proposition~\ref{prop-diminutive}, there is some complex $X\in\MLR_\mu$, and therefore some computable order~$h$ such that $X$ is contained in the $\Pi^0_1$~class ${\mathcal{P}=\{Y\in\cs\colon \K(Y\uh n)\geq h(n)\}}$.  Since $\mathcal{P}$ contains no computable points, it contains no $\mu$-atoms, and hence $\mu(\mathcal{P})=0$.  It follows from Theorem \ref{thm-kucera} that $\mathcal{P}\cap\MLR_\mu=\emptyset$, which contradicts the statement that $X\in\mathcal{P}\cap\MLR_\mu$. Thus $\mu$ must be diminutive.
\end{proof}

We now establish a connection between initial segment complexity and trivial measures.

\begin{prop}\label{prop-ac-trivial}
Suppose $\mu$ is a computable measure with the property that every ${X\in\MLR_\mu\setminus\atoms_\mu}$ is anti-complex.  Then $\mu$ is trivial.
\end{prop}

\begin{proof}
Suppose that $\mu$ is not trivial.  Then $\mu(\MLR_\mu\setminus\atoms_\mu)>0$.  If $\Phi$ is an almost total Turing functional such that~$\mu=\lambda_\Phi$, which exists by Theorem \ref{thm-measures-tf}(ii), then 
$
\lambda(\Phi^{-1}(\MLR_\mu\setminus\atoms_\mu))=\mu(\MLR_\mu\setminus\atoms_\mu)>0$.

Since the collection of 3-random sequences (that is, those sequences that are Martin-L\"of random relative to the oracle $\emptyset''$)  has Lebesgue measure $1$, it follows that $\Phi^{-1}(\MLR_\mu\setminus\atoms_\mu)$ contains a 3-random sequence.  By Theorem \ref{thm-ac-high}, every anti-complex, proper sequence is high.  However, no 3-random is high by Kautz~\cite[Theorem III.2.3]{Kau91}. As highness is closed upwards under Turing reducibility, this yields a contradiction.  Thus $\mu$ must be trivial.
\end{proof}

Using techniques similar to those used in Section \ref{sec-non-complex}, one can construct a computable measure $\mu$ where every ${X\in\MLR_\mu\setminus\atoms_\mu}$ is anti-complex.  However, the converse of Proposition \ref{prop-ac-trivial} does not hold, as there are computable, trivial measures $\mu$ such that 
$\MLR_\mu=\{A\}\cup\atoms_\mu$, where $A$ is low and hence not anti-complex by Theorem \ref{thm-ac-high}; such an example can be obtained, for instance, by applying Theorem 3.2 from Porter~\cite{Por15} to a low Martin-L\"of random sequence.

We conclude by showing that not every computable diminutive measure is trivial.

\begin{thm}
There is a computable diminutive measure that is not trivial.
\end{thm}

\begin{proof}
By Kautz'~\cite[Theorem IV.2.4]{Kau91} analysis of Martin's proof that the set of sequences of hyperimmune degree has Lebesgue measure~$1$, there is  an oracle Turing machine $\phi$ such that
$\dom(\phi)$ is a~$\Pi^0_2$~class and such that, for all $X\in\dom(\phi)$, $\phi^X$ is not dominated by a computable function. Then again by Kautz~\cite[Lemma II.1.4(ii)]{Kau91}, $\S$~contains a~$\Pi^{0,\emptyset'}_1$~class~$\P \subseteq \dom(\phi)$ with $\mu(\mathcal{P})>2^{-j}$ for some~$j$.
Let $(\U_i^{\emptyset'})_{i\in\omega}$ be a universal $\emptyset'$-Martin-L\"of test.  Then $\Q=\cs\setminus\U_j^{\emptyset'}$ is a $\Pi^{0,\emptyset'}_1$ class that satisfies $\lambda(\U_j^{\emptyset'})>1-2^{-j}$. Then $\mathcal{R}=\P\cap\Q$ is a $\Pi^{0,\emptyset'}_1$ class as well.
Observe that $\phi^X$ is total for all  $X \in \mathcal{R}$; that for all  $X \in \mathcal{R}$, $\phi^X$ is not dominated by a computable function; that $\lambda(\mathcal{R})>0$; that $\mathcal{R} = \llbracket T \rrbracket$ for some $\emptyset'$-computable tree $T$; and finally that $\mathcal{R}$ only contains 2-random sequences.

Since $\mathcal{R}$ is the set of paths through a $\emptyset'$-computable tree $T$, we can write $T=\lim_s T_s$, where $(T_s)_{s\in\omega}$ is a uniformly computable sequence of computable trees.  
Then 
we can define a partial computable function $f\colon\str\rightarrow\omega$ by letting $f(\sigma)$ be the least stage $s$ such that $\sigma\uh i\in T_s$ for every $i\leq|\sigma|$.  

For $Z\in\cs$ write $Z=z_0z_1z_2\dotsc$ with $z_i\in \{0,1\}$ for all $i$. Then we define a total Turing functional~$\Xi\colon\cs\rightarrow\cs$ on input $Z\in\cs$ in terms of blocks of output $\xi_i^Z$ (written as $\xi_i$).  Let $\xi_{-1}=z_0$.  For each $i\geq 0$, $\xi_i$ will either be finite in length or undefined.  For $n> 0$, $\xi_{n}$ is defined as follows:
\[
\xi_{n}=
	\left\{
		\begin{array}{ll}
			1^{f(Z\uh n)}\,0\,z_n & \mbox{if } f(Z\uh n)\halts, \\
			\mbox{undefined} & \mbox{if } f(Z\uh k)\diverge \mbox{ for some $k\leq n$}.
		\end{array}
	\right.
\]
We then define $\Xi(Z)$ by
\[
\Xi(Z)=
	\left\{
		\begin{array}{ll}
			\xi_{-1}\xi_0\,\xi_1\,\xi_2\dotsc & \mbox{if } |\xi_i|<\infty \mbox{ for every } i\geq 0, \\
			\xi_{-1}\xi_0\,\xi_1\,\xi_2\dotsc\xi_{i-1}1^\omega & \mbox{if $i$ is least such that } \xi_i \mbox{ is undefined.}
		\end{array}
	\right.
\]

Note that if $f(Z\uh n)\diverge$ for some $k$, then $\Xi(Z)$ will have the form $\sigma 1^\omega$ for some $\sigma\in\str$. If on the other hand $f(Z\uh k)\halts$ for all $k \in \omega$, then 
$
\Xi(Z)=z_0\,1^{f(Z\uh 1)}\,0\,z_1\,1^{f(Z\uh 2)}\, 0\,z_2\dotsc
$
We point out that the construction of $\Xi$ is inspired by
the proof of Lemma 2.1 in Ng et al.~\cite{NgSteYan13}.
Clearly $\Xi$ is total, and hence $\Xi(\cs)$ is a $\Pi^0_1$ class.  Moreover, for every $X\in\Xi(\cs)$ either $X$ is computable or $X$ is Turing equivalent to some $Y\in\mathcal{R}$.  Let $S$ be a computable tree such that $[S]=\Xi(\cs)$.

Next we define a total Turing functional $\Lambda\colon\cs\rightarrow\cs$.  For each $Y\in\cs$, $\Lambda(Y)$ will be computed as a series of blocks~$\lambda_i$. For non-computable sequences $Y\in\Xi(\cs)$, which have the form
\begin{equation}\label{eqn:dim}
Y=b_0\,1^{t_1}\,0\,b_1\,1^{t_2}\,0\,b_2\,1^{t_2}\,0\dotsc,
\end{equation}
$\Lambda$ will yield a non-computable output; for all other sequences, $\Lambda$ will yield a computable output.  Given a sequence of the form~(\ref{eqn:dim}), the blocks $\lambda_i$ are successively computed by the following procedure, where, for $\ell \in \omega$, we denote by $m_\ell$ the maximal~$k$ such that $Y\uh \ell$ is long enough to contain the coding location of~$b_k$.

\medskip

\begin{enumerate}[leftmargin=1.2cm,noitemsep]
\item  $j:=0$

\item  $s:=0$

\item  For all $j'\in \omega$, $k_{j'}:=1$

\item  For all $j'\in \{-1\} \cup \omega  $, $\lambda_{j'}:=\varepsilon$

\item  $b_{-1}:=b_0$

\item  If $Y\uh s \notin S$, enter an infinite loop that appends $b_{j-1}^\omega$ to $\lambda_{j-1}$

\item  If $Y \uh k_j$ does not contain the coding location of $b_j$

\item \quad \quad Append $b_{j-1}$ to $\lambda_{j-1}$
 
\item \quad \quad Increment $k_j$

\item Else

\item \quad \quad If $\phi^{b_0 \dotsc b_{m_{k_j}}}(j)[m_{k_j}]\diverge$

\item \quad \quad  \quad \quad Append $b_{j-1}$ to $\lambda_{j-1}$

\item \quad \quad \quad \quad Increment $k_j$

\item \quad \quad Else

\item \quad \quad \quad \quad Append $b_j^{\sum_{\ell=0}^{j-1}|\lambda_\ell| + k_j}$ to $\lambda_j$

\item \quad \quad \quad \quad Increment $j$

\item \quad \quad End If

\item End If

\item Increment $s$

\item Goto 6
\end{enumerate}

\medskip

\noindent As a result we obtain that
\[
\lambda_{i}=
	\left\{
		\begin{array}{ll}
			\varepsilon & \mbox{if } |\lambda_{i'}|=\infty \mbox{ for some $i'<i$},\\
		
			b_i^\omega &  \mbox{else if $j$ in the algorithm converges to } i+1, \\
		
			b_i^{k'} \mbox{ for some } k' \geq k & \mbox{else if } \exists k\,\forall i' \leq i\colon	\phi^{b_0\dotsc b_{m_k}}(i')[k]\halts. \\		
			
		\end{array}
	\right.
\]
We then define $\Lambda(Y)$ by
\[
\Lambda(Y)=
	\left\{
		\begin{array}{ll}
			\lambda_{-1}\lambda_0\,\lambda_1\,\lambda_2\dotsc & \mbox{if } |\lambda_i|<\infty \mbox{ for every } i>0, \\
			\lambda_{-1}\lambda_0\,\lambda_1\,\lambda_2\dotsc\lambda_{i} & \mbox{if $i$ is least such that } |\lambda_i|=\infty.
		\end{array}
	\right.
\]

We now have three cases to consider. 

\smallskip

\noindent\emph{Case 1:} $Y\notin\Xi(\cs)$.  In this case, for some output block $\lambda_i$ we will have $\lambda_i=b_i^\omega$ due to line 6~in the algorithm.
Hence we will have $\Lambda(Y)=\sigma\,b^\omega$ for some $b\in\{0,1\}$, where $\sigma\in\str$ are some bits output before the first stage~$s$ with~$Y\uh s\notin S$.

\smallskip

\noindent\emph{Case 2:} $Y\in\Xi(\cs)$ but $Y$ is computable.  Then there is some $W\notin\mathcal{R}$ such that $\Xi(W)=Y$, and hence $Y$ must end in a tail of $1$'s.
Then $\Lambda$ will only be able to extract finitely many bits $b_0,\dotsc,b_k$ from~$Y$.  
Then for $j=k+1$ the ``If'' in line~7 in the algorithm will always give a positive answer and 
we will have $\lambda_k=b_k^\omega$.
Hence $\Lambda(Y)=\sigma\, b^\omega$ for some $\sigma\in\str$ and some~$b\in\{0,1\}$.

\smallskip

\noindent\emph{Case 3:}  $Y\in\Xi(\cs)$ and $Y$ is not computable.  Then by our construction there is an~$X\in\mathcal{R}$ such that~$\Xi(X)=Y$.  $X\in\mathcal{R}$~implies that $\phi^X$ is total and not dominated by any computable function.  Thus we will have
$
\Lambda(Y)=b_0^{k'_0}\,b_1^{k'_1}\,b_2^{k'_2}\dotsc
$
for a sequence $(k'_i)_{i \in \omega}$ that is non-decreasing by line~15 in the algorithm.
In addition, by our convention that the value~$\phi^X(n)$ is at most equal to the least $s$ such that $\phi_s^X(n)\halts$, we have $\phi^X(i')\leq k'_i$ for every $i'\leq i$.  It follows that the function $i\mapsto k'_i$ is not dominated by any computable function.  Hence $\Lambda(Y)$ is not computable.

Now let $\Phi=\Lambda\circ\Xi$ and set $\nu=\lambda_\Phi$.  
By construction, $\Phi$ maps no sequence in $\mathcal{R}$ to an atom, and since
${\nu(\Phi(\mathcal{R})) = \lambda(\mathcal{R})>0}$, we have
${\nu(\atoms_\nu)\geq 1-\lambda(\mathcal{R})<1}$, that is, $\nu$ is not trivial.

To establish that $\nu$ is diminutive, we prove that no $W\in\MLR_\nu\setminus\atoms_\nu$ is complex.  Given {$W\in\MLR_\nu$}, it follows from the no randomness ex nihilo principle that there is some $X\in\MLR$ such that $\Phi(X)=W$ (see, for instance, Bienvenu and Porter~\cite[Theorem 3.5]{BiePor12}).  Note that we must have $X\in\mathcal{R}$, since $X\notin\mathcal{R}$ implies that $\Xi(X)$ is computable, and hence so is $\Phi(X)=\Lambda(\Xi(X))=W$.  But $W\in\MLR_\nu\setminus\atoms_\nu$ implies that $W$ is not computable.  It follows from the above case distinction that if we write $X=x_0x_1x_2\dotsc$ with $x_i \in \{0,1 \}$, for $i \in \omega$, then
\[
W=x_0^{k'_0}\,x_1^{k'_1}\,x_2^{k'_2}\dotsc,
\]
where the non-decreasing function $i\mapsto k'_i$ is not dominated by any computable function.  If we merge identical blocks of bits in $W$, we can write $W$ as
\[
W=a_0^{\ell_0}\,a_1^{\ell_1}\,a_2^{\ell_2}\dotsc,
\]
where $a_i \in \{0,1\}$ and $a_i\neq a_{i+1}$ for every $i\in \omega$ and the function $i\mapsto \ell_i$ is not dominated by any computable function (since $i\mapsto k'_i$ is non-decreasing and is not dominated by any computable function).  One can now argue as in the proof of Theorem \ref{thm-ioc-ioac} that $W$ is not complex; as there is no new idea involved, we leave this to the reader.  It thus follows that $\nu$ is diminutive, and the proof is complete.
\end{proof}

In the above proof, we showed that for every $W\in\MLR_\nu\setminus\atoms_\nu$, there is some non-computable $X\in\mathcal{R}$ such that~$\Phi(X)=W$.  Moreover, $X$ is 2-random, and hence by the relativization of the preservation of randomness (Theorem~\ref{thm-rand-pres}), it follows that $Y$ is 2-random with respect to $\nu$.  Since no sequence that is 2-random with respect to a computable measure is~$\Delta^0_2$ (that is,~Turing reducible to~$\emptyset'$), we have an alternative (priority-free) proof of the following known result.

\begin{corollary}[Kautz~\cite{Kau91}]
There is a computable, non-trivial measure $\mu$ such that there is no $\Delta^0_2$, non-computable $X\in\MLR_\mu$.
\end{corollary}

\section*{Acknowledgements}

The authors would like to thank Wolfgang Merkle for helpful
discussions about the subjects of this article, and for detailed
comments on an earlier draft.

H\"{o}lzl was supported by a Feodor Lynen postdoctoral research
fellowship by the Alexander von Humboldt Foundation
and by grant
R146-000-184-112 (MOE2013-T2-1-062) of the Ministry of Education Singapore. Porter was supported by grant OISE-1159158 of the
National Science Foundation as part of the
International Research Fellowship Program and is supported by grant  48003 of the John Templeton Foundation as part of the project ``Structure and Randomness
in the Theory of Computation.'' The opinions expressed in this
publication are those of the authors and do not necessarily reflect the
views of the John Templeton Foundation.

\appendix
\renewcommand*{\thesection}{\Alph{section}}

\section{Appendix}\label{appendix}

\begin{proof}[Proof of Lemma \ref{lem-KA-K}]
Define a functional $\Psi$ such that $(\rho,\sigma\tau)\in\Psi$ if and only if there are strings ${\rho_0,\rho_1\in\str}$ such that
\begin{itemize}[noitemsep,topsep=0pt]
\item[(i)] $\rho=\rho_0\rho_1$;
\item[(ii)] $U(\rho_0)=\sigma$, where $U$ is the universal prefix-free machine; and
\item[(iii)] $(\rho_1,\tau)\in S_\Phi$, where $\Phi$ a Turing functional that induces the universal left-c.e.\ semi-measure.
\end{itemize}
Then if $\sigma^*$ is the length-lexicographically least string such that $U(\sigma^*)=\sigma$,
$
\llb\Psi^{-1}(\sigma\tau)\rrb\supseteq\llb\{\rho_0\rho_1\colon\rho_0=\sigma^*\;\&\;\rho_1\in\Phi^{-1}(\tau)\}\rrb
$.
Hence
$
\lambda(\llb\Psi^{-1}(\sigma\tau)\rrb)\geq\lambda(\llb\{\rho_0\rho_1\colon\rho_0=\sigma^*\;\&\;\rho_1\in\Phi^{-1}(\tau)\}\rrb)=2^{-\K(\sigma)}\lambda(\llb\Phi^{-1}(\tau)\rrb)
$.
By the universality of $M$ there is some~$e$ such that $2^e \cdot M(\sigma\tau)\geq \lambda(\llb\Psi^{-1}(\sigma\tau)\rrb)$, and we have
$
2^e \cdot M(\sigma\tau)\geq2^{-\K(\sigma)}\lambda(\llb\Phi^{-1}(\tau)\rrb)
$.
Taking the negative logarithm of both sides yields $\KA(\sigma\tau)\leq \K(\sigma)+\KA(\tau)+O(1)$.
\end{proof}

\begin{proof}[Proof of Lemma \ref{lem:gran0}]

(i) First, for $k\in\omega$, $g(g^{-1}(k))\geq k$, and hence by the hypothesis we have $f(g^{-1}(k))\geq g(g^{-1}(k))\geq k$. Thus $f^{-1}(k)\leq g^{-1}(k)$. Now if $f^{-1}(k)=n$, so that $f(n)\geq k$, it follows by hypothesis that $g(n+c)> f(n)\geq k$.  Thus, $g^{-1}(k)\leq n+c=f^{-1}(k)+c$, which yields $f^{-1}(k)\geq g^{-1}(k)-c$.  Hence, we have, for every $k$,
$
g^{-1}(k)-c\leq f^{-1}(k)\leq g^{-1}(k)
$.

\noindent (ii) Let $k=f^{-1}(n+c)$, so that $f(k)\geq n+c$.  Then, since $f(k)\leq g(k)+c$, it follows that $g(k)\geq n$. Thus, ${g^{-1}(n)\leq k=f^{-1}(n+c)}$.
\end{proof}

\begin{proof}[Proof of Lemma \ref{lem-complex-inverse}]
	
(i) By definition of $g^{-1}$, $g^{-1}(n)$ is the least $k$ with $g(k)\geq n$.  It thus follows that ${g(g^{-1}(n)-1)<n}$.  By the monotonicity of $\KA$ we have $\KA(X\uh n)\geq\KA(X\uh g(g^{-1}(n)-1))\geq h(g^{-1}(n)-1)$.

\smallskip

\noindent (ii) Since $f(f^{-1}(n))\geq n$, we have $\KA(X\uh f^{-1}(n))\geq f(f^{-1}(n))\geq n$.
\end{proof}

\begin{proof}[Proof of Proposition \ref{prop-complex-strcomplex}]
First, it is routine to show that $\KA(\sigma)\leq \K(\sigma)+O(1)$ for every $\sigma$, from which it follows that every strongly complex sequence is complex.  For the converse, suppose that $X\in\cs$ satisfies $\K(X\uh h(n))\geq n$ for some computable order~$h$ and every $n$.  It is known that for every $\sigma$, we have $\K(\sigma)\leq \KA(\sigma)+\K(|\sigma|)+O(1)$ (see, for example, Downey and Hirschfeldt~\cite[Theorem~4.5.3]{DowHir10}).  If we restrict this inequality to strings~$\sigma$ of length $h(n)$, we have $\K(\sigma)\leq \KA(\sigma)+\K(h(n))+O(1)$.  But then since $\K(h(n))\leq \K(n)+O(1)$ for every $n$, we have $\K(\sigma)\leq \KA(\sigma)+\K(n)+O(1)$ for every string $\sigma$ of length $h(n)$.  Thus for every $n$ we have
\[
\begin{array}{rcl}
n &\leq& \K(X\uh h(n))\\
&\leq &\KA(X\uh h(n))+\K(n)+O(1)\\
&\leq &\KA(X\uh h(n))+2\log(n)+O(1).
\end{array}
\]
Then $\KA(X\uh h(n))\geq n-2\log(n)-O(1)$, and we can apply Lemma \ref{lem-complex-inverse}(i) to conclude that $\KA(X\uh n)$ dominates some computable order.  Hence $X$ is strongly complex.
\end{proof}

\begin{proof}[Proof of Proposition \ref{prop:gran}]

Suppose that for every $c$ there is some $n$ such that
\begin{equation}\label{eq-gran}
\KA(X\uh g_\mu(n))<n-c.
\end{equation}
By definition of $g_\mu$, for every $\sigma\in\str$ such that $|\sigma|=g_\mu(n)$ we have $\mu(\sigma)<2^{-n}$, which further implies that $n<-\log\mu(\sigma)$.  Applying this latter fact  to Equation~(\ref{eq-gran}) with $\sigma=X\uh g_\mu(n)$, we can conclude that for every $c$ there is some $n$ such that
\[
\KA(X\uh g_\mu(n))<-\log\mu(X\uh g_\mu(n))-c.
\]
But this implies that $X\notin\MLR_\mu$ by the Levin-Schnorr Theorem~\ref{thm-levin-schnorr}, contradicting our assumption.  Lastly, by a~reasoning similar to that in the proof of Lemma \ref{lem-complex-inverse}(i), from the inequality $\KA(X\uh g_\mu(n))\geq n-c$ for every $n$ we can derive $\KA(X\uh n)\geq g^{-1}_\mu(n)-c$ for every $n$.
\end{proof}

\begin{proof}[Proof of Lemma \ref{lem:gran1}] To define $f$, given $n$, we search a $\langle k,s\rangle$ (where $\langle\cdot,\cdot\rangle\colon\omega^2\rightarrow\omega$ is a computable 
pairing function) with
\begin{equation}\label{sdjasdfgdfgdfg}
(\forall \sigma\in 2^k)\; \mu_s(\sigma)+2^{-s}<2^{-n} 
\quad \text{and} \quad (\exists \tau\in 2^k)\; \mu_s(\tau)-2^{-s}>2^{-(n+2)}.
\end{equation}
Clearly for such a $k$, we will have $g_\mu(n)\leq k< g_\mu(n+2)$, so we can let $f(n)=k$.  

It remains to argue that for every $n$ such a $\langle k,s\rangle$ exists. We first claim 
that $g_\mu(n+1)\geq g_\mu(n)+1$ for every $n$. Suppose otherwise.  Then there is some $n$ such that $g_\mu(n+1)=g_\mu(n)$.  Setting $k=g_\mu(n)-1$, there is some $\sigma\in 2^k$ such that ${\mu(\sigma)\geq 2^{-n}}$ but $\mu(\sigma0)<2^{-(n+1)}$ and $\mu(\sigma1)<2^{-(n+1)}$, which implies that ${\mu(\sigma)=\mu(\sigma0)+\mu(\sigma1)<2^{-(n+1)}+2^{-(n+1)}=2^{-n}}$,  a contradiction.
So there is always a~$k$ strictly between $g_\mu(n)$ and $g_\mu(n+2)$ that satisfies Equation~(\ref{sdjasdfgdfgdfg}) for some~$s$.
\end{proof}

We use the following lemma in the proofs of Lemma \ref{prop-non-complex-KA} and  Proposition~\ref{prop-non-complex-KA-2}.

\begin{lem}\label{lem-complex-inverse-2}
Let $X\in\cs$ and let $h$ be an order.  
\begin{itemize}[noitemsep,topsep=0pt]
\item[(i)] If $\KA(X\uh h(n))\leq n$ for almost every $n$, then $\KA(X\uh n)\leq h^{-1}(n)$ for almost every~$n$.
\item[(ii)] If $\KA(X\uh n)\leq h(n)$ for almost every $n$, then $\KA(X\uh h^{-1}(n)-1)< n$ for almost every ~$n$.
\end{itemize}
\end{lem}

\begin{proof}
(i) Since $h(h^{-1}(n))\geq n$, by the monotonicity of $\KA$,  we have for sufficiently large $n$ that
\[
\KA(X\uh n)\leq \KA(X\uh h(h^{-1}(n)))\leq h^{-1}(n).
\]
(ii) Since $h(h^{-1}(n)-1)< n$, we have for sufficiently large $n$ that
$
\KA(X\uh h^{-1}(n)-1)\leq h(h^{-1}(n)-1)< n$.
\end{proof}

\begin{proof}[Proof of Proposition \ref{prop-non-complex-KA}]
	
(i) For the right-to-left direction, since there is some $c$ such that $\KA(\sigma)\leq \K(\sigma)+c$ for every~$\sigma$, we have, for infinitely many $n$,
$
n\leq \KA(X\uh f(n))\leq \K(X\uh f(n))+c$.
Hence, $\K(X\uh f(n+c))\geq n$ for infinitely many $n$.  For the left-to-right direction, if $\K(X\uh f(n))\geq n$ for infinitely many $n$, by following the steps of the proof of the left-to-right direction of Proposition \ref{prop-complex-strcomplex}, one can show that there is a computable order~$g$ such that $\KA(X\uh n)\geq g(n)$ for infinitely many $n$.  Let $h(n)$ be a computable order satisfying $h^{-1}(n)< g(n)$ for almost every $n$.  We claim that $\KA(X\uh h(n))\geq n$ for infinitely many $n$.  If not, then $\KA(X\uh h(n)) \leq n$ for almost every $n$.  Then by Lemma \ref{lem-complex-inverse-2}(i), $\KA(X\uh n)\leq h^{-1}(n)< g(n)$ for almost every $n$, which yields the desired contradiction.

\smallskip

\noindent (ii) For the left-to-right direction, let $c$ satisfy $\KA(\sigma)\leq \K(\sigma)+c$ for every $\sigma\in\str$.  Then given a computable order~$f$, since $X$~is anti-complex, it follows that 
$\K(X\uh f(n+c))\leq n$ for almost every $n$.  Then $\KA(X\uh f(n+c))\leq \K(X\uh f(n+c))+c\leq n+c$ for almost every $n$. For the other direction, suppose that $X$ is not anti-complex.  Then 
it follows that $X$ is i.o.\ complex, and thus by part (i), it is not the case that for every computable order~$f$, $\KA(X\uh f(n))\leq n$ for almost every $n$.

\smallskip

\noindent (iii)  This follows from the fact that a sequence is i.o.\ anti-complex if and only if it is not complex if and only if it is not strongly complex by Proposition~\ref{prop-complex-strcomplex}, and
the characterization of strong complexity given by Corollary \ref{cor-sc}.
\end{proof}

\begin{proof}[Proof of Proposition \ref{prop-non-complex-KA-2}]
Part (i) follows from part (ii) and the fact that a sequence is i.o.\ complex if and only if it is not anti-complex.  (ii) For the left-to-right direction, let $f$ be an arbitrary computable order. Let $h$~be a computable order such that $h^{-1}\leq f$.  Then by assumption, $\KA(X\uh h(n))\leq n$ for almost every~$n$.  Hence by Lemma~\ref{lem-complex-inverse-2}(i),  for almost every~$n$,
$
\KA(X\uh n)\leq h^{-1}(n)\leq f(n)
$.
For the right-to-left direction, again let $f$ be an arbitrary computable order, and let $h$~be a computable order such that~$f(n)+1\leq h^{-1}(n)$.  By assumption, $\KA(X\uh n)\leq h(n)$ for almost every $n$, and so by our assumption on~$f$ and Lemma \ref{lem-complex-inverse-2}(ii), for almost every $n$,
$
\KA(X\uh f(n))\leq \KA(X\uh h^{-1}(n)-1)<n
$.
Lastly, Part~(iii) follows from Corollary~\ref{cor-sc} and the fact that a sequence is i.o\@.~anti-complex if and only if it is not complex.  
\end{proof}

\section*{References}

\bibliographystyle{elsarticle-num}
\bibliography{heidelberg}

\begin{thebibliography}{10}
\expandafter\ifx\csname url\endcsname\relax
  \def\url#1{\texttt{#1}}\fi
\expandafter\ifx\csname urlprefix\endcsname\relax\def\urlprefix{URL }\fi
\expandafter\ifx\csname href\endcsname\relax
  \def\href#1#2{#2} \def\path#1{#1}\fi

\bibitem{ZvoLev70}
A.~K. Zvonkin, L.~A. Levin, The complexity of finite objects and the basing of
  the concepts of information and randomness on the theory of algorithms,
  Uspehi Mat. Nauk 25~(6(156)) (1970) 85--127.

\bibitem{KjoMerSte11}
B.~Kjos-Hanssen, W.~Merkle, F.~Stephan, Kolmogorov complexity and the recursion
  theorem, Trans. Amer. Math. Soc. 363~(10) (2011) 5465--5480.

\bibitem{Kan70}
M.~I. Kanovi{\v{c}}, The complexity of the enumeration and solvability of
  predicates, Dokl. Akad. Nauk SSSR 190 (1970) 23--26.

\bibitem{Kau91}
S.~M. Kautz, Degrees of random sets, Ph.D. thesis, Cornell University (1991).

\bibitem{Soa87}
R.~Soare, Recursively enumerable sets and degrees, Perspectives in Mathematical
  Logic, Springer, Berlin, 1987.

\bibitem{Nie09}
A.~Nies, Computability and randomness, Vol.~51 of Oxford Logic Guides, Oxford
  University Press, 2009.

\bibitem{DowHir10}
R.~G. Downey, D.~R. Hirschfeldt, Algorithmic randomness and complexity,
  Springer, 2010.

\bibitem{Lev74}
L.~A. Levin, Laws on the conservation (zero increase) of information, and
  questions on the foundations of probability theory, Problemy Pereda\v ci
  Informacii 10~(3) (1974) 30--35.

\bibitem{Cha75}
G.~J. Chaitin, A theory of program size formally identical to information
  theory, J. Assoc. Comput. Mach. 22 (1975) 329--340.

\bibitem{Lev73}
L.~A. Levin, The concept of a random sequence, Dokl. Akad. Nauk SSSR 212 (1973)
  548--550.

\bibitem{Kur81}
S.~Kurtz, Randomness and genericity in the degrees of unsolvability, Ph{D}
  dissertation, University of Illinois at Urbana (1981).

\bibitem{BiePor12}
L.~Bienvenu, C.~P. Porter, Strong reductions in effective randomness, Theoret.
  Comput. Sci. 459 (2012) 55--68.

\bibitem{Mil11}
J.~S. Miller, Extracting information is hard: {A} {T}uring degree of
  non-integral effective {H}ausdorff dimension, Adv. Math. 226~(1) (2011)
  373--384.

\bibitem{HigHudSim14}
K.~Higuchi, W.~M.~P. Hudelson, S.~G. Simpson, K.~Yokoyama, Propagation of
  partial randomness, Ann. Pure Appl. Logic 165~(2) (2014) 742--758.

\bibitem{Rei08}
J.~Reimann, Effectively closed sets of measures and randomness, Ann. Pure Appl.
  Logic 156~(1) (2008) 170--182.

\bibitem{ReiSla07}
J.~Reimann, T.~A. Slaman, Randomness for continuous measures, Draft.

\bibitem{ReiSla08}
J.~Reimann, T.~A. Slaman, Measures and their random reals, Preprint available
  at http://arxiv.org/abs/0802.2705.

\bibitem{Dem87}
O.~Demuth, A notion of semigenericity, Comment. Math. Univ. Carolin. 28~(1)
  (1987) 71--84.

\bibitem{DemKuc87}
O.~Demuth, A.~Ku{\v{c}}era, Remarks on {$1$}-genericity, semigenericity and
  related concepts, Comment. Math. Univ. Carolin. 28~(1) (1987) 85--94.

\bibitem{Mil02}
J.~S. Miller, Pi-0-1 classes in computable analysis and topology, Ph.D. thesis,
  Cornell University (2002).

\bibitem{MR3411165}
A.~Ku{\v{c}}era, A.~Nies, C.~P. Porter, Demuth's path to randomness, Bull.
  Symb. Log. 21~(3) (2015) 270--305.

\bibitem{BarGreMon11}
G.~Barmpalias, N.~Greenberg, A.~Montalb{\'a}n, T.~A. Slaman, ${K}$-trivials are
  never continuously random, in: Proceedings of the 11th {A}sian {L}ogic
  {C}onference, 2011, pp. 51--58.

\bibitem{HolMer10}
R.~H{\"o}lzl, W.~Merkle, Traceable sets, in: Theoretical Computer Science,
  Springer, 2010, pp. 301--315.

\bibitem{HigKih14}
K.~Higuchi, T.~Kihara, On effectively closed sets of effective strong measure
  zero, Ann. Pure Appl. Logic 165~(9) (2014) 1445--1469.

\bibitem{FraGreSte13}
J.~N.~Y. Franklin, N.~Greenberg, F.~Stephan, G.~Wu, Anti-complex sets and
  reducibilities with tiny use, J. Symbolic Logic 78~(4) (2013) 1307--1327.

\bibitem{BieDow09}
L.~Bienvenu, R.~Downey, {Kolmogorov Complexity and Solovay Functions}, in: 26th
  International Symposium on Theoretical Aspects of Computer Science, Vol.~3 of
  Leibniz International Proceedings in Informatics (LIPIcs), Schloss Dagstuhl
  -- Leibniz-Zentrum fuer Informatik, Dagstuhl, Germany, 2009, pp. 147--158.

\bibitem{Bin08}
S.~Binns, {$\Pi^0_1$} classes with complex elements, J. Symbolic Logic 73~(4)
  (2008) 1341--1353.

\bibitem{Por12}
C.~P. Porter, Mathematical and philosophical perspectives on algorithmic
  randomness, Ph.D. thesis, University of Notre Dame (2012).

\bibitem{Por15}
C.~P. Porter, Trivial {M}easures are not so {T}rivial, Theory Comput. Syst.
  56~(3) (2015) 487--512.

\bibitem{NgSteYan13}
K.~M. Ng, F.~Stephan, Y.~Yang, L.~Yu, Computational aspects of the
  hyperimmune-free degrees, in: Proceedings of the 12th {A}sian {L}ogic
  {C}onference, World Scientific, Hackensack, NJ, 2013, pp. 271--284.

\end{thebibliography}

\end{document}